\documentclass[11pt]{amsart}

\usepackage[pdftex]{graphicx}
\usepackage{amssymb}
\usepackage{amsmath}
\usepackage{amsfonts}
\usepackage{amsthm}
\usepackage{amssymb}
\usepackage{mathrsfs}
\newtheorem{theorem}{Theorem}[section]

\newtheorem{proposition}[theorem]{Proposition}
\newtheorem{corollary}[theorem]{Corollary}

\newtheorem{lemma}[theorem]{Lemma}

\theoremstyle{definition}
\newtheorem{definition}[theorem]{Definition}
\theoremstyle{remark}
\newtheorem{remark}[theorem]{Remark}
\newtheorem{example}[theorem]{Example}

\newcommand{\ind}{\mathrm{ind}}
\newcommand{\nul}{\mathrm{nul}}
\newcommand{\dom}{\mathrm{dom}}
\newcommand{\area}{\mathrm{Area}}

\begin{document}

\title[Index and isoperimetric inequalities]{Index of minimal spheres and isoperimetric eigenvalue inequalities}
\begin{abstract}
In the present paper we use twistor theory in order to solve two problems related to harmonic maps from surfaces to Euclidean spheres $\mathbb{S}^n$. First, we propose a new approach to isoperimetric inequalities based on energy index. Using this approach we show that for any positive $k$, the $k$-th non-zero eigenvalue of the Laplacian on the real projective plane endowed with a metric of unit area, is maximized on the sequence of metrics converging to a union of $(k-1)$ identical copies of round sphere and a single round projective plane. This extends the results of P. Li and S.-T. Yau for $k=1$ (1982); N. Nadirashvili and A. Penskoi for $k=2$ (2018); and confirms the conjecture made in~\cite{KNPP}. Second, we improve the known lower bounds for the area index of minimal two-dimensional spheres and minimal projective planes in $\mathbb{S}^n$. In the course of the proof we establish a twistor correspondence for Jacobi fields, which could be of independent interest for the study of moduli spaces of harmonic maps.
\end{abstract}

\author[M. Karpukhin]{Mikhail Karpukhin}
\address{Department of Mathematics,
University of California, Irvine, 340 Rowland Hall, Irvine, CA 92697-3875
}
\email{mkarpukh@uci.edu}
\maketitle

\section{Introduction}

\subsection{Laplacian eigenvalues}
Let $(M,g)$ be a closed surface with Riemannian metric $g$. The Laplace-Beltrami operator, or simply the Laplacian, is a natural operator $\Delta_g$, defined in local coordinates by the formula
$$
\Delta_g u = -\frac{1}{\sqrt{|g|}}\frac{\partial}{\partial x^i}\left(\sqrt{|g|}g^{ij}\frac{\partial u}{\partial x^j}\right)
$$
 
 It is well-known that for closed surfaces the spectrum of this operator consists only of eigenvalues of finite multiplicities. Thus, they form a sequence
 $$
 0=\lambda_0(M,g)<\lambda_1(M,g)\leqslant\lambda_2(M,g)\leqslant \lambda_3(M,g)\leqslant\ldots,
 $$
where eigenvalues are written with multiplicities.

Furthermore, we define normalized eigenvalues 
$$
\bar\lambda_k(M,g) = \lambda_k(M,g)\area_g(M).
$$
Consider the quantity
$$
\Lambda_k(M) = \sup_g \bar\lambda_k(M,g).
$$
The problem of isoperimetric eigenvalue inequalities is concerned with finding the exact value of $\Lambda_k(M)$ for all pairs $\{M,k\}$ and understanding for which metrics the supremum is achieved. We refer to such metrics as {\em maximal} for $\bar\lambda_k$.

For $k=1$ Yang and Yau~\cite{YangYau} have shown that for an orientable surface of genus $\gamma$ one has
\begin{equation}
\label{YY:ineq}
\Lambda_1(M)\leqslant 8\pi\left[\frac{\gamma+3}{2}\right],
\end{equation}
where $[x]$ stands for the integer part of $x$. The non-orientable version of inequality~\eqref{YY:ineq} was proven in~\cite{KarpukhinNonOrientable}. Furthermore, it was shown by the author in~\cite{KarpukhinYY} that inequality~\eqref{YY:ineq} is strict for $\gamma>2$.

Inequality~\eqref{YY:ineq} in combination with the result of~\cite{KarpukhinNonOrientable} means that for all surfaces $M$ one has $\Lambda_1(M)<+\infty$. The exact values of $\Lambda_1(M)$ are known only for a few $M$.
\begin{itemize}
\item Hersch~\cite{Hersch}, 1970: $\Lambda_1(\mathbb{S}^2)=8\pi$ and the supremum is achieved only for the round metric of constant curvature.
\item Li and Yau~\cite{LiYau}, 1982: $\Lambda_1(\mathbb{RP}^2)=12\pi$ and the supremum is achieved only for the round metric of constant curvature.
\item Nadirashvili~\cite{NadirashviliTorus}, 1996: $\Lambda_1(\mathbb{T}^2)=\frac{8\pi^2}{\sqrt{3}}$ and the supremum is achieved only for the flat equilateral metric, see also~\cite{CKM}.
\item Jakobson, Nadirashvili, Polterovich~\cite{JNP}, 2005, and El Soufi, Giacomini, Jazar~\cite{EGJ}, 2006: the metric that achieves $\Lambda_1(\mathbb{KL})$ is characterised as the metric induced by the unique minimal immersion to $\mathbb{S}^n$ by the first eigenfunctions.
\item Nayatani, Shoda~\cite{NayataniShoda}, 2019: for an orientable surface $\Sigma_2$ of genus $2$ one has $\Lambda_1(\Sigma_2) = 16\pi$. The metric is induced on a Bolza surface by a hyperelliptic covering of $\mathbb{S}^2$. The metric on a Bolza surface was conjectured to be maximal in~\cite{JLNNP}. The authors of~\cite{JLNNP} confirmed their conjecture using numerical calculations.
\end{itemize}

For $k>1$ Korevaar~\cite{Korevaar} (see also~\cite{GY, Hassannezhad}) confirmed the conjecture of Yau~\cite[Problem 71]{Y} and proved that there exists a universal constant $C$, such that 
$$
\Lambda_k(M)\leqslant Ck(\gamma+1).
$$
Thus, $\Lambda_k(M)<+\infty$ for all $k$. Until very recently, the only known results for $k>1$ were $\Lambda_2(\mathbb{S}^2)$ (see~\cite{NadirashviliS2,PetridesS2}), $\Lambda_3(\mathbb{S}^2)$ (see~\cite{NadirashviliSire}) and $\Lambda_2(\mathbb{RP}^2)$ (see~\cite{NadirashviliPenskoi}).

In~\cite{KNPP} the author jointly with Nadirashvili, Penskoi and Polterovich has obtained the first general result that covers all values of $k$.
\begin{theorem}[K., Nadirashvili, Penskoi, Polterovich~\cite{KNPP}]
\label{KNPP:thm}
For all $k\geqslant 1$ one has $\Lambda_k(\mathbb{S}^2)=8\pi k$. Moreover, for any smooth metric $g$ on $\mathbb{S}^2$ and $k>1$ one has 
$$
\bar\lambda_k(\mathbb{S}^2,g)<8\pi k.
$$
The supremum in the definition of $\Lambda_k(\mathbb{S}^2)$ is achieved on a sequence of metrics converging to a union of $k$ touching round spheres.
\end{theorem}

Furthermore, in the same paper~\cite{KNPP} it was conjectured that the situation for the projective plane $\mathbb{RP}^2$ is similar to that of $\mathbb{S}^2$. The main result of the present paper is the proof of that conjecture. In particular, we prove the following theorem.

\begin{theorem}
\label{MainTheorem}
For all $k\geqslant 1$ one has $\Lambda_k(\mathbb{RP}^2)=4\pi(2k+1)$. Moreover,
for any smooth metric $g$ on $\mathbb{RP}^2$ and $k>1$ one has
\begin{equation}
\label{MainTheorem:ineq}
\bar\lambda_k(\mathbb{RP}^2,g)<4\pi(2k+1).
\end{equation}
The supremum in the definition of $\Lambda_k(\mathbb{S}^2)$ is achieved on a sequence of metrics converging to a union of $k-1$ identical round spheres and a standard projective plane touching each other, such that the ratio of the areas of the projective plane and the spheres is $3:2$.
\end{theorem}
\begin{remark}
The fact that the sequence of metrics described in the theorem saturates the bound~\eqref{MainTheorem:ineq} is a well-known fact going back to~\cite{CE}. The main statement of Theorem~\ref{MainTheorem} is the inequality~\eqref{MainTheorem:ineq} itself.
\end{remark}
\begin{remark}
Theorem~\ref{MainTheorem} was previously proved by Li and Yau in~\cite{LiYau} for $k=1$; and by Nadirashvili and Penskoi in~\cite{NadirashviliPenskoi} for $k=2$. The result for $k\geqslant 3$ is new.
\end{remark}
\begin{remark}
Our proof of Theorem~\ref{MainTheorem} is drastically different from that of Theorem~\ref{KNPP:thm} and requires novel ideas outlined in the next section.
\end{remark}

\subsection{Energy index of harmonic maps}
A map $\Phi\colon (M,g)\to (N,h)$ from a Riemannian surface to a Riemannian manifold is called {\em harmonic} if it is a critical point of the energy functional
$$
E_g(\Phi) = \frac{1}{2}\int_M|d\Phi|_g^2\,dv_g.
$$
{\em Energy index} $\ind_E(\Phi)$ is defined to be the Morse index of $\Phi$ as a critical point of $E_g(\Phi)$, see Definition~\ref{indE:def} for a more precise formulation.

There is a classical connection established in~\cite{NadirashviliTorus, ESI} between extremal metrics for eigenvalue functionals and harmonic maps $\Phi\colon (M,g)\to \mathbb{S}^n\subset\mathbb{R}^{n+1}$, where $\mathbb{S}^n$ is equipped with the standard round metric. Let us briefly recall this connection in case $M=\mathbb{S}^2$, for more details see Section~\ref{indSextremal:sec} below. It follows from the Euler-Lagrange equation for the harmonic map that the pull-backs of coordinate functions on $\mathbb{R}^{n+1}$ are eigenfunctions with eigenvalue $2$ of the operator $\Delta_{g_\Phi}$, where $g_\Phi = \frac{1}{2}|d\Phi|^2_gg$. The {\em spectral index} $\ind_S(\Phi)$ is defined to be the smallest $k$ such that $\lambda_k(M,g_\Phi)=2$, see Definition~\ref{indS:def} for a more precise formulation. The results of~\cite{ESI} state that the metrics $g_\Phi$ are extremal metrics for the functionals $\bar\lambda_{\ind_S(\Phi)}(\mathbb{S}^2,g)$ and, moreover, an inverse statement holds with some modifications.

In the present paper we introduce a novel method of studying $\Lambda_k(M)$. The main idea is the interplay between the two notions of index of a harmonic map $\Phi$, $\ind_S(\Phi)$ and $\ind_E(\Phi)$. The following proposition holds for any surface $M$, not necessarily $\mathbb{S}^2$ or $\mathbb{RP}^2$.
\begin{proposition}
\label{MainTheorem:ESintro}
Let $\Phi\colon (M,g)\to\mathbb{S}^n$ be a linearly full harmonic map, i.e. the image $\Phi(M)$ is not contained in any proper equatorial subsphere. Then one has 
$$
\ind_S(\Phi)\geqslant \frac{\ind_E(\Phi)}{n+1}.
$$ 
\end{proposition}

The proof of this proposition is fairly straightforward and is inspired by~\cite[Section 6]{FS}. In case $M=\mathbb{S}^2$ or $M=\mathbb{RP}^2$ one can prove a more general statement, see Theorem~\ref{indnul:thm}. The proof of Theorem~\ref{indnul:thm} is more complicated and uses twistor theory, see Section~\ref{twistor:sec}.

The main reason our proof of Theorem~\ref{MainTheorem} works is that for $M=\mathbb{S}^2$ we are able to prove an inequality in the opposite direction.

\begin{theorem}
\label{MainTheorem:SEintro}
Let $\Phi\colon\mathbb{S}^2\to\mathbb{S}^n$ be a linearly full harmonic map. Then
$$
(n-2)\ind_S(\Phi)\leqslant \ind_E(\Phi).
$$
\end{theorem}
We believe that the inequalities of this type have a variety of applications not only to spectral theory but also to the study of minimal surfaces and harmonic maps. We demonstrate this by improving the known upper bounds for the energy index of harmonic spheres in $\mathbb{S}^n$.

Recall~\cite{Barbosa,Calabi} that for any linearly full harmonic map $\Phi\colon\mathbb{S}^2\to\mathbb{S}^n$ one has $n=2m$ and $E_g(\Phi) = 4\pi d$, where $d$ is the integer referred to as {\em harmonic degree} (or, simply, {\em degree}) of $\Phi$.
\begin{theorem}
\label{MainTheorem:indintro}
Let $\Phi\colon\mathbb{S}^2\to\mathbb{S}^{2m}$ be a linearly full harmonic map of degree $d$. Then one has the following inequality
$$
\ind_E(\Phi)\geqslant 2(m-1)(2d - [\sqrt{8d+1}]_{\mathrm{odd}}+2),
$$
where $[x]_{\mathrm{odd}}$ denotes the largest odd number not exceeding $x$.
\end{theorem}
\begin{remark}
\label{harmmin:rem}
Any harmonic map $\mathbb{S}^2\to\mathbb{S}^{2m}$ is weakly conformal and, therefore, is a branched minimal immersion. In particular, the area index coincides with the energy index. Thus, Theorem~\ref{MainTheorem:indintro} gives the same bound for the area index of branched minimal spheres in $\mathbb{S}^n$.
\end{remark}

Let us compare Theorem~\ref{MainTheorem:indintro} to similar results in the literature. By Remark~\ref{harmmin:rem}, the energy index coincides with the area index. The existing literature on the subject of area index of minimal submanifolds in spheres is largely concerned with unbranched hypersurfaces. The only general results in codimension greater than $1$ we were able to find are~\cite{MU, EjiriIndex}. In~\cite{MU} the authors show that the index of any 
linearly full unbranched minimal sphere of degree $d$ in $\mathbb{S}^4$ is equal to $4d-2$. In~\cite{EjiriIndex} the author proved that the index of any linearly full unbranched minimal sphere of degree $d$ in $\mathbb{S}^{2m}$ is at least $4d+2(m-3)$. The bound of Theorem~\ref{MainTheorem:indintro} is weaker than both these results in case $m=2$, but is stronger than the bound of~\cite{EjiriIndex} once $m>2$ and holds for branched surfaces as well.

For any harmonic map $\Phi\colon\mathbb{RP}^2\to\mathbb{S}^{2m}$, its pre-composition with an antipodal projection yields a harmonic map $\widetilde\Phi\colon\mathbb{S}^2\to\mathbb{S}^{2m}$. We use twistor correspondence to show the relation between $\ind_E(\Phi)$ and $\ind_E(\widetilde\Phi)$.

\begin{proposition}
\label{RP2:thmintro}
Let $\Phi\colon\mathbb{RP}^2\to\mathbb{S}^{2m}$ be a linearly full harmonic map and let $\widetilde\Phi\colon\mathbb{S}^2\to\mathbb{S}^{2m}$ be its antipodal lift. Then one has
$$
\ind_E(\Phi) = \frac{1}{2}\ind_E(\widetilde \Phi).
$$
\end{proposition}

Finally, we make a remark about our proof of Theorem~\ref{MainTheorem:SEintro} that could be of interest to specialists in twistor theory. One of the main technical steps in our proof of Theorem~\ref{MainTheorem:SEintro} is a first order twistor correspondence for harmonic maps to $\mathbb{S}^{2m}$ which we prove in Section~\ref{1twistor:sec}. In other words, we show that each Jacobi field can be lifted to a vector field on the twistor space satisfying certain properties. We call these lifts {\em twistor fields}. This is an extension of the results in~\cite{LW}, where the authors proved a similar first order correspondence for harmonic spheres in $\mathbb{S}^4$.

\subsection{Plan of the proof} Let us outline the plan and the main ideas of the proof of Theorem~\ref{MainTheorem}. First, we use existence theory for maximal metrics developed in~\cite{Petrides, Petridesk} to reduce Theorem~\ref{MainTheorem} to an inequality between the degree and the spectral index of a harmonic map $\Phi\colon \mathbb{RP}^2\to\mathbb{S}^{2m}$. To prove this inequality, we apply Theorem~\ref{MainTheorem:indintro} to the lift $\widetilde\Phi$ of $\Phi$ and, finally, combine Propositions~\ref{MainTheorem:ESintro} and~\ref{RP2:thmintro}.

Theorem~\ref{MainTheorem:SEintro} is the main ingredient that allows this plan to come to fruition. Its proof relies on the geometry of the moduli space of harmonic maps $\mathbb{S}^2\to\mathbb{S}^{2m}$. In particular, the moduli space admits an action of the group $SO(2m+1,\mathbb{C})$ defined via twistor correspondence. Our main observation is that this action preserves the energy index. This observation is inspired by the work of Ejiri~\cite{Ejiri}, where  $SO(2m+1,\mathbb{C})$-invariance of the spectral index is proved. As a result, a choice of appropriate elements of $SO(2m+1,\mathbb{C})$ leads to the proof of Theorem~\ref{MainTheorem:SEintro} by induction on $m$.

\subsection*{Organization of the paper} The paper is organized as follows. In Section~\ref{indSextremal:sec} we recall the background on the the theory of maximal metrics, define spectral index and reduce Theorem~\ref{MainTheorem} to an inequality between the degree and the spectral index. Section~\ref{Energy:sec} is devoted to energy index and inequalities between indices. Theorem~\ref{MainTheorem:ESintro} is proved in Section~\ref{indSE:sec}. In Sections~\ref{indEI:sec} and~\ref{indEII:sec} we show Theorems~\ref{MainTheorem:SEintro} and~\ref{MainTheorem:indintro}. Finally, Section~\ref{proof:sec} contains the proof of Theorem~\ref{MainTheorem}.

The proofs of some of the auxiliary statements in Section~\ref{Energy:sec} are left to Section~\ref{twistor:sec}.
In particular, in Section~\ref{Jacobi:sec} we prove Theorem~\ref{RP2:thmintro} and in Section~\ref{1twistor:sec} we establish the first order twistor correspondence.
\subsection*{Acknowledgements} The author is grateful to A.~Fraser, A.~Penskoi, I.~Polterovich and R.~Schoen for fruitful discussions. The author thanks V.~Medvedev and I.~Polterovich for remarks on the preliminary version of the manuscript.



\section{Spectral index and extremal metrics}

\label{indSextremal:sec}

\subsection{Preliminaries on harmonic maps from the sphere and projective plane}
Let $(M,g)$ be a Riemannian surface and let $\Phi\colon (M,g)\to\mathbb{S}^n$ be a harmonic map. It is equivalent to the following equation, see e.g.~\cite[Example 4.14]{EL},
\begin{equation}
\label{harmonic:equation}
\Delta_g\Phi=|\nabla\Phi|_g^2\Phi,
\end{equation}
where $\Delta_g$ is applied component-wise. The energy $E_g(\Phi)$ of $\Phi$ satisfies
$$
E_g(\Phi) = \frac{1}{2}\int_M|\nabla\Phi|^2_g\,dv_g = \area_{\left(\frac{1}{2}|\nabla\Phi|_g^2\,g\right)}(M).
$$
The map $\Phi$ is called {\em linearly full} if its image is not contained in a proper equatorial subsphere of $\mathbb{S}^n$. Recall that the energy is invariant under conformal change of metric on $M$ and, therefore, the harmonicity of $\Phi$ depends only on the conformal class of the metric $g$.

If $M=\mathbb S^2$ or $M=\mathbb{RP}^2$, then up to a diffeomorphism there is a unique conformal class of metrics on $M$. Thus, without loss of generality we may assume that in either case $g$ is the standard round metric and, furthermore, we will often suppress the metric $g$ in the notations. The theory of harmonic maps $\mathbb{S}^2\to\mathbb{S}^n$ is a classical subject of differential geometry. Below we recall some results of this theory, for more detailed exposition see Section~\ref{twistor:sec} below.
\begin{theorem}[Calabi~\cite{Calabi}, Barbosa~\cite{Barbosa}]
\label{Barbosa:thm}
Let $\Phi\colon \mathbb{S}^2\to\mathbb{S}^n$ be a linearly full harmonic map. Then the following holds,
\begin{itemize}
\item $\Phi$ is weakly conformal, i.e. $\Phi$ is a branched minimal immersion;
\item $n=2m$;
\item $E(\Phi) = 4\pi d$, where $d$ is a natural number satisfying $d\geqslant \frac{m(m+1)}{2}$.
\end{itemize}
The number $d$ is referred to as {\em degree} of the harmonic map $\Phi$. 
\end{theorem}
\begin{remark}
Sometimes $d$ is referred to as {\em twistor} or {\em harmonic} degree of $\Phi$.
\end{remark}

Barbosa also defined the action of the group $SO(2m+1,\mathbb{C})$ on the space of all linearly full harmonic maps $\mathbb{S}^2\to\mathbb{S}^{2m}$. We will recall the definition of this action in Section~\ref{soaction:sec}, for now it is sufficient to know that this action preserves the degree of $\Phi$.

Assume now that we have a harmonic map $\Phi\colon \mathbb{RP}^2\to\mathbb{S}^n$. If $\pi\colon\mathbb{S}^2\to\mathbb{RP}^2$ is an antipodal projection, we let $\widetilde\Phi\colon\mathbb{S}^2\to\mathbb{S}^n$ be the composition $\widetilde\Phi = \Phi\circ\pi$. Applying Theorem~\ref{Barbosa:thm} to $\widetilde \Phi$ and noting that $2E(\Phi) = E(\widetilde \Phi)$, one obtains the following theorem.

\begin{theorem}
\label{BarbosaRP:thm}
Let $\Phi\colon \mathbb{RP}^2\to\mathbb{S}^n$ be a linearly full harmonic map. Then
\begin{itemize}
\item $\Phi$ is weakly conformal, i.e. $\Phi$ is a branched minimal immersion;
\item $n=2m$;
\item $E(\Phi) = 2\pi d$, where $d$ is a natural number satisfying $d\geqslant \frac{m(m+1)}{2}$.
\end{itemize}
We call $d$ the {\em degree} of the harmonic map $\Phi$.
\end{theorem}

\begin{example}
\label{Veronese1}
{\em Veronese maps.} We recall that the spectrum of $\Delta$ on the round $\mathbb{S}^2$ of unit radius is given by numbers $l_m=m(m+1)$, each repeated with multiplicity $2m+1$. The corresponding eigenspaces $\mathcal E_m$ consist of harmonic polynomials of degree $m$ in $\mathbb{R}^3$ restricted to $\mathbb{S}^2$. It turns out (see e.g.~\cite{Calabi}) that there exists a basis $\{f^m_j\}$ of $\mathcal E_m$ such that $\Phi_m = (f^m_1,\ldots, f^m_{2m+1})$ is a linearly full harmonic map $\Phi_m\colon\mathbb{S}^2\to\mathbb{S}^{2m}$. Moreover, the degree $\Phi_m$ saturates the lower bound in Theorem~\ref{Barbosa:thm}, i.e. $\deg(\Phi_m) = \frac{m(m+1)}{2} = \frac{1}{2}l_m$. If $m$ is even, then the map $\Phi_m$ induces a map $\Psi_m\colon\mathbb{RP}^2\to\mathbb{S}^{2m}$, which saturates the lower bound of Theorem~\ref{BarbosaRP:thm}.
\end{example}

In the paper~\cite{EjiriEquivariant} Ejiri studied whether there are further restrictions on harmonic maps from the projective plane. In particular, he proved the following theorem.
\begin{theorem}[Ejiri~\cite{EjiriEquivariant}]
\label{EjiriRP:thm}
Let $\Phi\colon \mathbb{RP}^2\to\mathbb{S}^{2m}$ be a linearly full harmonic map. Then the following holds,
\begin{itemize}
\item $m$ is even;
\item If $\deg(\Phi) = \frac{m(m+1)}{2}$, then up to an isometry of $\mathbb{S}^{2m}$, $\Phi$ coincides with the Veronese map $\Psi_m$.
\end{itemize}
\end{theorem}
\begin{remark}
In~\cite{G} it is shown that any harmonic map $\Phi\colon \mathbb{RP}^2\to\mathbb{S}^{4}$ has odd degree. It would be interesting to see if a similar statement holds for spheres of higher dimension.
\end{remark}

\subsection{Spectral index} To each harmonic map $\Phi\colon (M,g)\to\mathbb{S}^n$ one can associate the following Schr\"odinger operator, see e.g.~\cite{Ejiri},
$$
\mathcal L_{g,\Phi} (u) = \Delta_g u - |\nabla\Phi|_g u
$$
acting on functions on $M$. This operator is conformally covariant in the sense that for a conformal metric $\widetilde g = e^{2\omega}g$ one has $\mathcal L_{\widetilde g,\Phi} = e^{-2\omega}\mathcal L_{g,\Phi}$. The associated quadratic form is given by
$$
Q_{\Phi,S}(u) = \int_M|\nabla u|_g^2 - |\nabla\Phi|_g^2u^2\,dv_g
$$
and is conformally invariant. As a result, the number of negative eigenvalues of $\mathcal L_{g,\Phi}$ as well as the kernel $\ker \mathcal L_{g,\Phi}$ do not depend on the choice of a metric in a conformal class. 

\begin{definition}
\label{indS:def}
{\em Spectral index} $\ind_S(\Phi)$ is a number of negative eigenvalues of $\mathcal L_{\widetilde g,\Phi}$ for some (any) metric $\widetilde g$ conformal to $g$. Similarly, {\em spectral nullity} $\nul_S(\Phi)$ is the dimension $\ker \mathcal L_{\widetilde g,\Phi}$ for some (any) metric $\widetilde g$ conformal to $g$.
\end{definition}
\begin{remark}
Equation~\eqref{harmonic:equation} implies that the components of $\Phi$ lie in $\ker \mathcal L_{g,\Phi}$, so that for a linearly full harmonic map $\Phi\colon (M,g)\to \mathbb{S}^n$ one has $\nul_S(\Phi)\geqslant n+1$.
\end{remark}
\begin{remark}
In~\cite{MR,Nayatani,Ejiri} these numbers are simply referred to as index and nullity. We decided to add the epithet "spectral" to distinguish it from the energy index defined below. 
\end{remark}
\begin{remark}
The operator $\mathcal L_{g,\Phi}$ appears naturally in various geometric contexts, e.g. in the study of complete minimal surfaces in $\mathbb{R}^3$, see e.g.~\cite{MR,Nayatani}.
\end{remark}
Let us now make a particular choice of metric. Set $g_\Phi = \frac{1}{2}|\nabla\Phi|^2_gg$, then $\mathcal L_{g_{\Phi}, \Phi} = \Delta_{g_\Phi} - 2$. Then the components of $\Phi$ are eigenfunctions of $\Delta_{g_\Phi}$ with the eigenvalue $2$. As a result, $\ind_S(\Phi)$ is the number of eigenvalues of $\Delta_{g_{\Phi}}$ smaller than $2$ and $\nul_S(\Phi)$ is the multiplicity of the eigenvalue $2$. We remark that, in general, $g_\Phi$ is a metric with isolated conical singularities at branch points of $\Phi$. However, it does not affect the discussion above as long as one works with the Friedrichs extension of the Laplacian, see e.g.~\cite[Section 2.3]{KarpukhinYY}. Finally, if $\Phi$ is weakly conformal then $g_\Phi = \Phi^*g_{\mathbb{S}^n}$.

\begin{example} 
\label{Veronese2}
{\em Veronese surfaces.} Using the definition of Veronese surfaces $\Phi_m$ and $\Psi_m$ in Example~\ref{Veronese1} it is easy to compute their spectral index. Indeed, 
$$
\ind_S(\Phi_m) = \sum_{i=0}^{m-1} (2i+1) = m^2.
$$
If $m$ is even, then
$$
\ind_S(\Psi_m) = \sum_{i=0}^{\frac{m}{2}-1}(4i+1) = \frac{m(m-1)}{2}.
$$
\end{example}

\subsection{Extremal metrics for $\bar\lambda_k$}
In the present section we review the theory of extremal metrics for functionals $\bar\lambda_k(M,g)$ briefly mentioned in the introduction. We set
$$
\Lambda_k(M,[g]) = \sup_{\tilde g\in [g]}\bar\lambda_k(M,\tilde g),
$$
where $[g]$ is a class of metrics conformal to $g$. Similar to the previous section it is convenient to allow $g$ to have conical singularities at isolated points of $M$. Thus, $[g] = \{\tilde g|\,\tilde g = f^2g\}$, where $f$ ranges over smooth functions with isolated zeroes.

The functional $\bar\lambda_i(M,g)$ depends continuously  on the metric $g$, but this functional is not differentiable. 
However, Berger proved in the paper~\cite{Berger} that for an analytic family of metrics $g_t$ there exist the left and the right 
derivatives with respect to $t$. This motivates the following definition, see the papers~\cite{ESI,NadirashviliTorus}.

\begin{definition} A Riemannian metric $g$ on a closed surface $M$ is called {\em extremal for the functional} 
$\bar\lambda_i$ if for any analytic deformation $g_t$ such that $g_0 = g$ the following inequality holds,
$$
\frac{d}{dt}\bar\lambda_i(M,g_t)\Bigl|_{t=0+} \times \frac{d}{dt}\bar\lambda_i(M,g_t)\Bigl|_{t=0-}\leqslant 0.
$$

Similarly, $g$ is called {\em conformally extremal} if the same inequality holds for conformal deformations, i.e. for deformations satisfying $[g_t]=[g]$ for all $t$.
\end{definition}


\begin{theorem}[Nadirashvili~\cite{NadirashviliTorus}, El Soufi and Ilias,~\cite{ESI}, see also~\cite{KNPP}] 
\label{extremal:thm}
If $g$ is conformally extremal for the functional $\bar\lambda_i(M,g)$ then there exists a harmonic map $\Psi\colon (M,g)\to\mathbb{S}^n$ whose components are $\lambda_i$-eigenfunctions. In particular, $\ind_S(\Phi)\leqslant i$. If $g$ is extremal, then $\Psi$ can be chosen to be weakly conformal.

Conversely, if $\Phi\colon(M,g)\to\mathbb{S}^n$ is a harmonic map to the unit sphere, then the metric $g_\Phi$ is conformally extremal for the functional $\bar\lambda_{\ind_S(\Phi)}$. Furthermore, if $\Phi$ is weakly conformal, i.e. if $\Phi\colon (M,g_{\Phi})\to\mathbb{S}^n$ is a branched minimal immersion, then $g_\Phi$ is extremal for the functional $\bar\lambda_{\ind_S(\Phi)}$.
\end{theorem}

In particular, if there exists a metric that realizes the quantities $\Lambda_k(M)$ or $\Lambda_k(M,[g])$, then this metric is extremal or conformally extremal respectively. Such metrics are called {\em maximal for} $\bar\lambda_k$ or {\em conformally maximal for} $\bar\lambda_k$ respectively. The existence of such metrics was studied in~\cite{NadirashviliTorus, Petrides, Petridesk}. For the sake of brevity we only state the result for $M=\mathbb{RP}^2$.
\begin{theorem}[Petrides~\cite{Petridesk}, see also~\cite{KNPP}]
\label{Petrides:thm}
Assume that 
\begin{equation}
\label{Petrides:ineq}
\Lambda_k(\mathbb{RP}^2)>\Lambda_{k-1}(\mathbb{RP}^2) + 8\pi,
\end{equation}
then there exists a metric $g$ on $\mathbb{RP}^2$, smooth outside of possible isolated conical singularities, such that $\bar\lambda_k(\mathbb{RP}^2,g) = \Lambda_k(\mathbb{RP}^2)$. In particular, there exists a harmonic map $\Phi\colon \mathbb{RP}^2\to\mathbb{S}^n$ of spectral index at most $k$, such that $g=g_\Phi$.

\end{theorem}
\begin{remark}
It follows from the results of~\cite{CE} that if inequality~\eqref{Petrides:ineq} fails, then 
$\Lambda_k(\mathbb{RP}^2)=\Lambda_{k-1}(\mathbb{RP}^2) + 8\pi$.
\end{remark}
\begin{remark}
We refer the reader to~\cite{Petridesk} for the case of an arbitrary surface $M$.
\end{remark}

%

\subsection{Relation between spectral index and degree.}

The majority of the present paper is devoted to the proof of the following theorem.
\begin{theorem}
\label{MainTheorem:ind}
Let $\Phi\colon \mathbb{RP}^2\to\mathbb{S}^{n}$ be a harmonic map of degree $d$. Then one has
$$
\ind_S(\Phi)\geqslant\dfrac{d-1}{2}.
$$
Moreover, the equality is achieved iff $d=3$.
\end{theorem}
In this section we show that Theorem~\ref{MainTheorem:ind} implies Theorem~\ref{MainTheorem}.
\begin{proof}[Proof of Theorem~\ref{MainTheorem}]
The proof is by induction on $k$. For $k=1$ Theorem~\ref{MainTheorem} holds by the results of Li and Yau~\cite{LiYau}.

By the step of induction, $\Lambda_k(\mathbb{RP}^2)=4\pi(2k+1)$. Assume that $\Lambda_{k+1}(\mathbb{RP}^2)>4\pi(2k+3)=\Lambda_{k}(\mathbb{RP}^2)+8\pi$, then by Theorem~\ref{Petrides:thm} the value  $\Lambda_{k+1}(\mathbb{RP}^2)$ is attained for some metric $g$, smooth outside finitely many conical singularities. According to Theorem~\ref{extremal:thm} there exists a harmonic map $\Phi\colon \mathbb{RP}^2\to\mathbb{S}^n$ such that $\ind_S(\Phi)\leqslant k+1$, $g=g_\Phi$ and $\bar\lambda_{k+1}(\mathbb{RP}^2,g) = \Lambda_{k+1}(\mathbb{RP}^2)$. Let $d$ be the degree of $\Phi$, then one has $4\pi(2k+3)<\bar\lambda_{k+1}(\mathbb{RP}^2,g) =2\area_{g_\Phi}(\mathbb{RP}^2) = 2E_g(\Phi)= 4\pi d$. Thus, one has $d>2k+3$. At the same time, $\ind_S(\Phi) \leqslant k+1$ and by Theorem~\ref{MainTheorem:ind} one has
$$
k+1\geqslant \ind_S(\Phi)\geqslant\frac{d-1}{2}>\frac{2k+2}{2} = k+1.
$$  
We arrive at a contradiction.

Assume now that for some $k>1$ and some metric $g$ one has the equality $\Lambda_k(\mathbb{RP}^2)=\bar\lambda_k(\mathbb{RP}^2,g)=4\pi(2k+1)$. Then $g$ is maximal and by Theorem~\ref{extremal:thm} there exists a harmonic map $\Phi$, such that $\deg(\Phi)=2k+1$ and $\ind_S(\Phi)\leqslant k$. Since $k>1$, one has $d:=\deg(\Phi)>3$ and the inequality in Theorem~\ref{MainTheorem:ind} is strict. Thus, one obtains
$$
k\geqslant \ind_S(\Phi)>\frac{d-1}{2}=\frac{2k}{2} = k,
$$  
where once again we arrive at a contradiction.
\end{proof}

\begin{example}
\label{Veronese3}
By Examples~\ref{Veronese1} and~\ref{Veronese2}, for Veronese immersions $\Psi_m\colon\mathbb{RP}^2\to\mathbb{S}^{2m}$ one has $\deg(\Psi_m) = \frac{m(m+1)}{2}$ and $\ind_S(\Psi_m) = \frac{m(m-1)}{2}$. Then
$$
2\ind_S(\Psi_m) +1  = m(m-1) +1 \geqslant \frac{m(m+1)}{2},
$$
 where the last inequality holds because by Theorem~\ref{EjiriRP:thm} one has $m\geqslant 2$. Moreover, the equality is achieved only for $m=2$ and in that case $\deg(\Psi_m)=3$. As a result, we see that Theorem~\ref{MainTheorem:ind} holds for Veronese immersions.
\end{example}



\section{Energy index}
\label{Energy:sec}

\subsection{Energy index} If $\Phi\colon (M,g)\to (N,h)$ is a harmonic map between Riemannian manifolds, the {\em Jacobi} (or {\em index}) operator is defined on sections of $\Phi^*TN$ by the following formula, see e.g.~\cite{EL},
$$
J_{g,\Phi}(V) = \Delta^\Phi_g V - \mathrm{tr}_g R^N(V,d\Phi) d\Phi,
$$
where $\Delta^\Phi_g$ is a positive Laplacian associated to the induced connection on $\Phi^*TN$ and $R^N$ is the Riemann curvature tensor of $N$, $R^N(X,Y)Z = \nabla_X\nabla_YZ-\nabla_Y\nabla_XZ-\nabla_{[X,Y]}Z$. In the following we assume $M$ is a surface and $N=\mathbb{S}^n$ is a Euclidean sphere of unit radius. In this case the Jacobi operator $J_g$ has the following coordinate representation. The bundle $\Phi^*T\mathbb{S}^n$ can be identified with a subbundle $\underline{\Phi}^\perp$ of the trivial bundle $M\times \mathbb{R}^{n+1}$ consisting of vectors orthogonal to $\Phi$. Then one has
$$
J_{g,\Phi}(V) = \pi_{\underline{\Phi}^\perp}\left(\Delta_g V - |\nabla \Phi|^2_g V\right),
$$ 
where $\Delta_g$ is a Laplacian on $M$ applied component-wise and $\pi_{\underline{\Phi}^\perp}$ is the orthogonal projection onto $\underline{\Phi}^\perp$. Note that $J_{g,\Phi}$ is conformally covariant in the same 
way as $\mathcal L_{g,\Phi}$. In the following we sometimes omit the subscript $\Phi$ if the map can be inferred from context.

The associated quadratic form $Q_{\Phi,E}$ is given by 
$$
Q_{\Phi,E}(V) = \sum_{i=1}^{n+1}\int_M\left (|\nabla V^i|^2_g-|\nabla\Phi|_g^2(V^i)^2\right)\,dv_g
$$
and is conformally invariant, which is why omit the index $g$ in the notation. Note that the domain $\dom Q_{\Phi,E}$ of $Q_{\Phi,E}$ consists of $\mathbb{R}^{n+1}$-valued functions $V$ on $M$ satisfying $V\cdot\Phi=0$, i.e. for each point $x\in M$ one has $V(x)\perp\Phi(x)$.

Similarly to the spectral index, we define the {\em energy index}.
\begin{definition}
\label{indE:def}
{\em Energy index} $\ind_E(\Phi)$ is a number of negative eigenvalues of $J_{\widetilde g,\Phi}$ for some (any) metric ${\widetilde g}$ conformal to $g$. Similarly, {\em energy nullity} $\nul_E(\Phi)$ is the dimension $\ker J_{\widetilde g,\Phi}$ for some (any) metric ${\widetilde g}$ conformal to $g$. The fields in $\ker J_{\widetilde g,\Phi}$ are referred to as {\em Jacobi fields} along $\Phi$.
\end{definition}
\begin{remark}
\label{Jacobi:remark}
One of the ways to produce Jacobi fields along $\Phi$ is to consider harmonic deformations of $\Phi$. Assume that $\Phi_t$ is a family of harmonic maps, such that $\Phi = \Phi_0$. Then $X=\frac{d}{dt}|_{t=0}\Phi_t$ is a Jacobi field, see e.g.~\cite{LW}. Jacobi fields obtained in this way are called~{\em integrable}. The question of determining whether all Jacobi fields are integrable is quite subtle and is related to the smoothness of the corresponding moduli space, see~\cite{LWcp, LW} for results in this direction.
\end{remark}

\subsection{Bounds between indices}
\label{indSE:sec}
The main idea of the present paper is the interplay between $\ind_S(\Phi)$ and $\ind_E(\Phi)$ for a harmonic map $\Phi$. Indeed, looking at the definition of the corresponding quadratic forms, one sees that $Q_{\Phi,E}(V)=\sum Q_{\Phi,S}(V^i)$. However, the domain $\dom Q_{\Phi,E}$ complicates the considerations.

\begin{theorem}
Let $\Phi\colon (M,g)\to\mathbb{S}^n$ be a linearly full harmonic map. Then one has
\begin{equation}
\ind_S(\Phi)\geqslant \frac{\ind_E(\Phi)}{n+1}.
\end{equation}
\end{theorem}
\begin{proof}
Let $U^-$ be the negative space of $Q_S$, i.e. $\dim U^- = \ind_S(\Phi)$. Assume the contrary, i.e. $(n+1)\ind_S(\Phi)<\ind_E(\Phi)$. Then there exists a vector-function $X\in\dom Q_E$, such that $Q_{E}(X)<0$ and
$$
\int_M Xf\,dv_g = 0
$$
for all $f\in U^-$, i.e. all components $X^i$, $i=1,\ldots,n+1$ are orthogonal to $U^-$. Thus, $Q_S(X^i)\geqslant 0$ and one concludes
$$
0>Q_{E}(X) = \sum_{i=1}^{n+1}Q_{S}(X^i)\geqslant 0.
$$
This contradiction completes the proof.
\end{proof}

\begin{theorem}
\label{indnul:thm}
Let $\Phi\colon (M,g)\to\mathbb{S}^{2m}$ be a linearly full harmonic map, where $M$ is either $\mathbb{S}^2$ or $\mathbb{RP}^2$. Then one has
\begin{equation}
\label{indnul:ineq}
\ind_S(\Phi)\geqslant \frac{\ind_E(\Phi) +\nul_E(\Phi)-m(2m+1)}{2m+1}
\end{equation}
\end{theorem}

\begin{proof}
The proof is very similar to the previous one. The idea is to make use of Jacobi fields. The new ingredient is the following proposition, whose proof is postponed until Section~\ref{jacobi:sec}.
\begin{proposition}
\label{infinitesimal isometry}
Let $\Phi\colon (M,g)\to\mathbb{S}^{2m}$ be a linearly full harmonic map, where $M$ is either $\mathbb{S}^2$ or $\mathbb{RP}^2$. Let $X$ be a Jacobi field, such that 
$$
\Delta_g X^i = |\nabla\Phi|_g^2 X^i,
$$
i.e. for all $i=1,\ldots, 2m+1$ one has $X^i\in\ker(\mathcal L_\Phi)$.
Then $X$ is trivial, in a sense that there exists a constant matrix $A\in \mathfrak{so}(2m+1,\mathbb{R})$ such that $X=A\Phi$.
\end{proposition}
\begin{remark} In~\cite{Barbosa} it is shown that if two linearly full harmonic maps $\Phi_1,\Phi_2\colon\mathbb{S}^2\to\mathbb{S}^{2m}$ have the same energy density, then they differ by an isometry of $\mathbb{S}^{2m}$. Proposition~\ref{infinitesimal isometry} is an infinitesimal version of this fact.
\end{remark}

Assume the contrary to inequality~\eqref{indnul:ineq}. Let $U^-$ be the negative subspace of $Q_S$. Since $\dim \mathfrak{so}(2m+1,\mathbb{R}) = m(2m+1)$, there exists a non-trivial vector field $X$ (i.e. not of the form $A\Phi$ for $A\in \mathfrak{so}(2m+1,\mathbb{R})$) such that $Q_{E}(X)\leqslant 0$ and 
$$
\int_M Xf\,dv_g = 0
$$
for all $f\in U^-$, i.e. all components $X^i$, $i=1,\ldots,2m+1$ are orthogonal to $U^-$. Thus, $Q_{S}(X^i)\geqslant 0$ and one concludes
$$
0\geqslant Q_{E}(X) = \sum_{i=1}^{n+1}Q_{\Phi}(X^i)\geqslant 0.
$$
Therefore, $X^i\in\ker(\mathcal L_\Phi)$ and, thus, $X$ is a Jacobi field. We arrive at a contradiction with Proposition~\ref{infinitesimal isometry}.
\end{proof}

\subsection{Bound for the energy index I}
\label{indEI:sec}

The goal of this section is to prove Theorem~\ref{MainTheorem:SEintro}. Let us restate it.
\begin{theorem}
\label{index:thm}
Let $\Phi\colon \mathbb{S}^2\to\mathbb{S}^{2m}$ be a linearly full harmonic map. Then
$$
\ind_E(\Phi)\geqslant 2(m-1)\ind_S(\Phi).
$$
\end{theorem}

The idea is to use the action of the group $SO(2m+1,\mathbb{C})$ on the space of linearly full harmonic maps $\mathbb{S}^2\to\mathbb{S}^{2m}$ defined by Barbosa~\cite{Barbosa} and described in Section~\ref{soaction:sec}. In order to do that, one needs to understand how this action affects $\ind_S$ and $\ind_E$. For the spectral index this problem was solved by Ejiri.
\begin{theorem}[Ejiri~\cite{Ejiri}, see also~\cite{Kotani}]
\label{Ejiri:thm}
For every linearly full harmonic map $\Phi\colon \mathbb{S}^2 \to \mathbb{S}^{2m}$ there exists a $1$-parameter subgroup $A_t\subset SO(2m+1,\mathbb{C})$ such that harmonic maps $\Phi_t = A_t\Phi$ satisfy the following properties,
\begin{itemize}
\item[1)] harmonic maps $\Phi_t$ converge in compact-open topology to a harmonic map $\Phi_\infty$ into a subsphere $\mathbb{S}^{2(m-1)}\subset \mathbb S^{2m}$.
\item[2)] for all $t$ one has $\nul_S(\Phi_t)=\nul_S(\Phi) = \nul_S(\Phi_\infty)$ and, as a result, $\ind_S(\Phi_t)=\ind_S(\Phi) = \ind_S(\Phi_\infty)$. 
\end{itemize} 
\end{theorem}
\begin{remark}
\label{convergence:remark}
Technically speaking, in~\cite{Ejiri} Ejiri showed the convergence of the corresponding twistor lifts. However, it is known that twistor correspondence is a homeomorphism, see e.g.~\cite{FernandezTopology}. Furthermore, harmonic maps $\mathbb{S}^2\to\mathbb{S}^{2m}$ are analytic and, thus, convergence in compact-open topology implies $C^\infty$-convergence. 
\end{remark}

Ejiri used this theorem to show that the following holds.
\begin{corollary}
\label{Ejiri:cor}
For every harmonic map $\Phi\colon \mathbb{S}^2 \to \mathbb{S}^{2m}$ of index $d>1$ one has $\ind_S(\Phi)\geqslant d+1$.
\end{corollary}

In the present paper we show that the action of $SO(2m+1,\mathbb{C})$ preserves the energy index as well.

\begin{proposition}
\label{invariance:prop}
The action of $SO(2m+1,\mathbb{C})$ on the moduli space of linearly full harmonic maps $\mathbb{S}^2\to\mathbb{S}^{2m}$ preserves $\nul_E$.
\end{proposition}
The proof is postponed until Section~\ref{soaction:sec}.

\begin{corollary}
\label{invariance:cor}
For $\Phi_t$ defined as in Theorem~\ref{Ejiri:thm} one has $\nul_E(\Phi_t)=\nul_E(\Phi) \leqslant \nul_E(\Phi_\infty)$. As a result, $\ind_E(\Phi_t)=\ind_E(\Phi) \geqslant \ind_E(\Phi_\infty)$.
\end{corollary}
\begin{proof}
Remark~\ref{convergence:remark} implies continuity of eigenvalues of the Jacobi operator $J_\Phi$ with respect to $\Phi$. Thus, the function $\ind_E+\nul_E$ is upper semicontinuous whereas the function $\ind_E$ is lower semicontinuous. Thus, the assertion easily follows from Proposition~\ref{invariance:prop}.
\end{proof}

\begin{proof}[Proof of Theorem~\ref{index:thm}]
Similarly to the argument of Ejiri in~\cite{Ejiri}, we apply Theorem~\ref{Ejiri:thm} repeatedly. Starting with $\Phi$ we obtain a map $\Phi_\infty$ with the image inside a proper subsphere.

\begin{proposition}
Let $\Psi\colon (M,g)\to\mathbb{S}^n$ be a harmonic map with the image in a $k$-dimensional proper equatorial subsphere. Let $\widetilde\Psi$ be the same harmonic map considered as a linearly full map to $\mathbb{S}^k$. Then $\ind_E(\Psi) = (n-k)\ind_S(\Psi) + \ind_E(\widetilde\Psi)$.
\end{proposition}
\begin{proof}
Up to an orthogonal transformation of $\mathbb{R}^{n+1}$ we can assume that the image of $\Psi$ lies in the subspace  defined by $x^{k+2}=\ldots=x^{n+1}=0$. Then the domain of $J_{g,\Psi}$ decomposes $\dom(J_{g,\Psi})=\dom(J_{g,\widetilde \Psi}) \oplus C^\infty(M,V^\perp)$, where the second summand corresponds to smooth $V^\perp$-valued vector-functions. Moreover, on $C^\infty(M,V^\perp)$ the operator $J_{g,\Psi}$ acts as $\mathcal L_{g,\Psi}$ applied to each component separately. Therefore, $\ind_E(\Psi) = \ind_E(\widetilde\Psi) + (n-k)\ind_S(\Psi)$.
\end{proof}
On the next step we apply Theorem~\ref{Ejiri:thm} to a linearly full map $\widetilde\Phi_{\infty}$. On each step we are decreasing the dimension $m$ and after finitely many steps we obtain a map $\Psi\colon\mathbb{S}^2\to \mathbb{S}^{2m}$ with the image in a two-dimensional subsphere such that $\ind_E(\Phi)\geqslant 2(m-1)\ind_S(\Psi) + \ind_E(\widetilde \Psi)\geqslant 2(m-1)\ind_S(\Psi)$ and $\ind_S(\Phi) = \ind_S(\Psi)$. 

To finish the proof, we recall that by Theorem~\ref{Ejiri:thm}, $\ind_S(\Psi) = \ind_S(\Phi)$. One obtains
$$
\ind_E(\Phi) \geqslant 2(m-1)\ind_S(\Psi) = 2(m-1)\ind_S(\Phi),
$$
which concludes the proof of Theorem~\ref{index:thm}.
\end{proof}

\subsection{Bound on spectral index}

In order to effectively use the bound of Theorem~\ref{index:thm} one needs to obtain a lower bound on $\ind_S(\Phi)$. It turns out that for our purposes, the bound of Corollary~\ref{Ejiri:cor} is not sufficient. In this section we modify the methods of Ejiri to show the following proposition.
\begin{proposition}
\label{indS:prop}
Let $\Phi\colon \mathbb{S}^2\to\mathbb{S}^{2m}$ be a harmonic map of degree $d$. Then 
\begin{equation}
\label{ind:eqtn}
\ind_S(\Phi)\geqslant 2d - \nul_S(\Phi)+2.
\end{equation}
\end{proposition}
\begin{proof}
Similarly to the proof of Theorem~\ref{index:thm}, one can apply Theorem~\ref{Ejiri:thm} repeatedly to obtain a harmonic map $\Phi$ with the same spectral index and spectral nullity, whose image lies in a two-dimensional subsphere. Therefore, it suffices to prove the proposition for $m=1$.

The case of $m=1$ was studied in detail for e.g. in~\cite{EKo,MR, Nayatani}. In that case $\Phi$ is either holomorphic or anti-holomorphic and harmonic degree coincides with the usual definition of degree for (anti-)holomorphic maps. Up to conjugation we assume that $\Phi$ is a holomorphic map of degree $d$. Let $M_d$ be the space of all such maps. It is easy to see that it is a connected complex manifold. 
\begin{proposition}[Montiel, Ros~\cite{MR}; Ejiri, Kotani~\cite{EKo}]
\label{indexS2:prop}
The set $N_d=\{\Phi\in M_d\,|\, \nul_S(\Phi)>3\}$ is a proper analytic subset of $M_d$. Moreover, for all $\Phi\in M_d\setminus N_d$ one has $\ind_S(\Phi) = 2d-1$.
\end{proposition} 

Proposition~\ref{indexS2:prop} proves the inequality~\eqref{ind:eqtn} when $\nul_S(\Phi)=3$. Assume $\nul_S(\Phi)>3$, then $\Phi\in N_d$. Since $N_d$ is a proper analytic subset, there exists an analytic deformation $\Phi_t\in M_d$ such that $\Phi_0=\Phi$ and $\Phi_t\not\in N_d$ for $t\ne 0$. Since the eigenvalues of a $\mathcal L_\Phi$ are continuous with respect to $\Phi$ (see Remark~\ref{convergence:remark}), the function $\ind(\Phi_t)+\nul(\Phi_t)$ is an upper-semicontinuous function of $t$. As a result, 
$$
\ind_S(\Phi)+\nul_S(\Phi)\geqslant\lim\limits_{t\to 0}\left(\ind_S(\Phi_t)+\nul_S(\Phi_t) \right)= 2d+2.
$$
\end{proof}
\begin{remark}
We remark that the bound~\eqref{ind:eqtn} is likely not optimal. In~\cite[p. 116]{Ejiri} Ejiri suggests that one should always have $\ind_S(\Phi) = 2d - \frac{1}{2}(\nul_S(\Phi)-1)$.
\end{remark}

\subsection{Bounds on nullities}
Both Theorem~\ref{indnul:thm} and Proposition~\ref{indS:prop} refer to $\nul_E$ and $\nul_S$ respectively. In this section we collect the information about these quantities necessary for our proof of Theorem~\ref{MainTheorem:ind}.
\begin{proposition}[Ejiri~\cite{Ejiri}, Kotani~\cite{Kotani}]
For any harmonic map $\Phi\colon\mathbb{S}^2\to\mathbb{S}^{2m}$ the spectral nullity $\nul_S(\Phi)$ is odd. 
\end{proposition}
\begin{proposition}[Kotani~\cite{Kotani}]
\label{nulS:prop}
Let $\Phi\colon \mathbb{S}^2\to\mathbb{S}^{2m}$ be a harmonic map of degree $d$. Then one has
\begin{equation}
\label{Kotani:ineq}
d\geqslant \frac{\nul_S(\Phi)^2-1}{8}.
\end{equation}
Equivalently, using notations of Theorem~\ref{MainTheorem:indintro},
$$
\nul_S(\Phi)\leqslant [\sqrt{8d+1}]_{\mathrm{odd}}.
$$
\end{proposition}
\begin{proof}[Idea of the proof]
Without loss of generality $\Phi$ is linearly full.
Recall that $\nul_S(\Phi)\geqslant 2m+1$. In~\cite{Kotani} Kotani proves that if $\nul_S(\Phi)= 2m+2\nu+1$, then $\Phi$ can deformed to a linearly full map $\Psi\colon\mathbb{S}^2\to\mathbb{S}^{2m+2\nu}$ of the same degree. Thus, the inequality~\eqref{Kotani:ineq} follows from Theorem~\ref{Barbosa:thm}.
\end{proof}

\begin{theorem}[Fernandez~\cite{Fernandez}, Kotani~\cite{Kotani}]
\label{nulE:thm}
Let $\Phi\colon\mathbb{S}^2\to\mathbb{S}^{2m}$ be a linearly full harmonic map of degree $d$. Then one has
$$
\nul_E(\Phi)\geqslant 4d+2m^2.
$$
\end{theorem}
\begin{proof}
In both papers~\cite{Fernandez,Kotani}, the authors show that the space of linearly full harmonic maps $\mathbb{S}^2\to\mathbb{S}^{2m}$ of degree $d$ is a complex algebraic variety of pure complex dimension at least $2d+m^2$. Therefore, at each point the tangent cone to the moduli space is of real dimension at least $4d+2m^2$. Since tangent cone is a set of equivalence classes of analytic arcs, by Remark~\ref{Jacobi:remark}, the real dimension of the cone of {\em integrable} Jacobi fields is at least $4d+2m^2$. Thus, the linear span of this cone forms a subspace in a space of Jacobi fields of dimension at least $4d+2m^2$.
\end{proof}
\begin{remark}
We remark that Fernandez~\cite{Fernandez} proved that the complex dimension is exactly $2d+m^2$. It does not imply, however, that one has equality in Theorem~\ref{nulE:thm}.
\end{remark}

\subsection{Bound on energy index II} 
\label{indEII:sec}
To prove Theorem~\ref{MainTheorem:indintro} one combines Theorem~\ref{index:thm} with Propositions~\ref{indS:prop} and~\ref{nulS:prop} to arrive exactly at 
$$
\ind_E(\Phi)\geqslant 2(m-1)(2d-[\sqrt{8d+1}]_{\mathrm{odd}}+2).
$$

Finally, we state a slightly more general version of Proposition~\ref{RP2:thmintro}, whose proof we once again postpone until Section~\ref{Jacobi:sec}.
\begin{theorem}
\label{RP2index:thm}
Let $\Phi\colon \mathbb{RP}^2\to\mathbb{S}^{2m}$ be a linearly full harmonic map and let $\widetilde\Phi\colon\mathbb{RP}^2\to\mathbb{S}^{2m}$ be a lift of $\Phi$ via the antipodal projection. Then
$$
\ind_E(\Phi) = \frac{1}{2}\ind_E(\widetilde\Phi),\qquad \nul_E(\Phi) = \frac{1}{2}\nul_E(\widetilde\Phi).
$$
\end{theorem}

\subsection{Proof of Theorem~\ref{MainTheorem:ind}}
\label{proof:sec}
Without loss of generality, assume that $\Phi\colon\mathbb{S}^2\to\mathbb{S}^{2m}$ is linearly full.
By Theorem~\ref{index:thm}, Proposition~\ref{indS:prop}, Theorem~\ref{nulE:thm} and Theorem~\ref{indnul:thm} one has 
\begin{equation}
\label{final:eq1}
\ind_S(\Phi)\geqslant \frac{(2d-\nul_S(\Phi)+2)(m-1)+2d-m-m^2}{2m+1}.
\end{equation}
At the same time, $\nul_S(\Phi)\geqslant 2m+1$ and $8d\geqslant \nul_S(\Phi)^2-1$ by Proposition~\ref{nulS:prop}.
We claim that the combination of these inequalities with Theorem~\ref{EjiriRP:thm} yields Theorem~\ref{MainTheorem:ind} unless $(m,d) \in\{ (2,3),(2,4), (4,10)\}$. The statement is purely computational.
\begin{proof}
Assume that the assertion of Theorem~\ref{MainTheorem:ind} fails, i.e. $\frac{d-1}{2}\geqslant \ind_S(\Phi)$. Combining it with~\eqref{final:eq1} one obtains
$$
\frac{d-1}{2}\geqslant \frac{(2d-\nul_S(\Phi)+2)(m-1)+2d-m-m^2}{2m+1}.
$$
After elementary transformations, the inequality becomes
\begin{equation}
\label{final:eq2}
(2m-2)\nul_S(\Phi) + 2m^2+3-4m\geqslant (2m-1)d.
\end{equation}
It follows that 
$$
(2m-2)\nul_S(\Phi) + m(2m-2)\geqslant (2m-1)d +2m-3\geqslant (2m-2)d,
$$
where we used that for maps from $\mathbb{RP}^2$ one has $m\geqslant 2$.
Cancelling the factor $(2m-2)$ yields
$$
\nul_S(\Phi) + m\geqslant d.
$$
Combining with~\eqref{Kotani:ineq} one obtains
\begin{equation}
\label{final:eq3}
\nul_S(\Phi) + m\geqslant \frac{\nul_S(\Phi)^2-1}{8}.
\end{equation}
Solving this quadratic inequality yields
$$
\nul_S(\Phi)\leqslant \frac{8+\sqrt{64+4(8m+1)}}{2}\leqslant 8+\sqrt{8m+1},
$$ 
where we used the inequality $\sqrt{a+b}\leqslant \sqrt a + \sqrt b$. Combining this with~\eqref{final:eq3} one obtains,
$$
8+\sqrt{8m+1}\geqslant \nul_S(\Phi)\geqslant d-m\geqslant \frac{m(m-1)}{2},
$$
where we are using $d\geqslant \frac{m(m+1)}{2}$. Taking into account $m\geqslant 2$ is an integer, the last inequality holds only for $m\leqslant 6$. Then by Theorem~\ref{EjiriRP:thm},  $m=2,4,6$. For these values of $m$ we work directly with inequality~\eqref{final:eq2}. 

For $m=2$, combining~\eqref{final:eq2} with inequality~\eqref{Kotani:ineq} yields
\begin{equation}
\label{final:eq4}
2\nul_S(\Phi)+3\geqslant 3d\geqslant \frac{3(\nul_S(\Phi)^2-1)}{8},
\end{equation}
which yields $\nul_S(\Phi)\leqslant 6$. Since $\nul_S(\Phi)$ is odd and $\nul_S(\Phi)\geqslant 2m+1=5$, one has $\nul_S(\Phi)=5$. Then~\eqref{final:eq4} yields $13\geqslant 3d\geqslant 9$, i.e. $d=3,4$.

By similar arguments for $m=4$ one has
 \begin{equation*}
6\nul_S(\Phi)+19\geqslant 7d\geqslant \frac{7(\nul_S(\Phi)^2-1)}{8},
\end{equation*}
which similarly implies $\nul_S(\Phi)=9$. Then $73\geqslant 7d\geqslant 70$, i.e. $d=10$.

For $m=6$ one obtains 
$$
10\nul_S(\Phi) + 51\geqslant 11d\geqslant \frac{11(\nul_S(\Phi)^2-1)}{8},
$$
which yields $\nul_S(\Phi)\leqslant 10$ and contradicts $\nul_S(\Phi)\geqslant 2m+1=11$.
\end{proof}
Finally, let us deal with the exceptional cases.
\begin{itemize}
\item[1)] Assume $d=3$. By Theorem~\ref{EjiriRP:thm}, $\Phi$ is a Veronese immersion. Then Theorem~\ref{MainTheorem:ind} claims $\ind_S(\Phi)=1$. This is shown in Example~\ref{Veronese2}, see also~\cite{LiYau}.
\item[2)]
Assume $d=4$. Then Theorem~\ref{MainTheorem:ind} claims $\ind_S(\Phi)>1$. By the result of Li and Yau~\cite{LiYau}, the only map of spectral index $1$ has degree $3$, therefore for a map of degree $4$, one has $\ind_S(\Phi)>1$.
\item[3)]
Assume $m=4$ and $d=10$. Then $d=\frac{m(m+1)}{2}$ and by Theorem~\ref{EjiriRP:thm}, up to an isometry $\Phi$ is a Veronese immersion. In this case the theorem follows from Example~\ref{Veronese3}, where we checked Theorem~\ref{MainTheorem:ind} for Veronese immersions.
\end{itemize}



\section{Twistor correspondence}
\label{twistor:sec}
In this section we first recall the twistor correspondence~\cite{Calabi,Barbosa} with an emphasis on the notion of harmonic sequence~\cite{CW1, CW2}. For the purposes of our exposition, one should think of the harmonic sequence as a convenient setup for coordinate computations. Theorems~\ref{RP2index:thm} and~\ref{indnul:thm} are proved in Sections~\ref{Jacobi:sec} and~\ref{jacobi:sec} respectively. The first order twistor correspondence is established in Section~\ref{1twistor:sec}. In Section~\ref{soaction:sec} we recall the action of $SO(2m+1,\mathbb{C})$ and use the first order twistor correspondence to prove Proposition~\ref{invariance:prop}.

\subsection{Harmonic sequence} 
In the present section we discuss the concept of {\em harmonic sequence} introduced by Chern and Wolfson in~\cite{CW1,CW2}, see also~\cite{EW}. Below, we follow the presentation in~\cite{BJRW}. Recall that the property of being harmonic depends only on the conformal class $[g]$ of the metric on a Riemannian surface $M$. In order to define a harmonic sequence one requires a complex structure on $M$, therefore, we require $M$ to be orientable and fix a choice of orientation on $M$. A surface $M$ with a conformal class and an orientation can be endowed with a canonical complex structure, making $M$ a Riemann surface. We will discuss the dependence on orientation in Remark~\ref{orientation:remark} below.

Let $L\subset\mathbb{CP}^n\times\mathbb{C}^{n+1}$ be a tautological bundle over $\mathbb{CP}^n$, i.e $L = \{(l,v)\,|v\in l\}$. Let $M$ be a Riemann surface. There is a correspondence between smooth maps $\psi\colon M\to\mathbb{CP}^n$ and line subbundles of a trivial bundle $M\times \mathbb{C}^{n+1}\to M$ given by $\psi\leftrightarrow\psi^*L$. We endow $M\times\mathbb{C}^{n+1}$ with the usual Hermitian inner product $\langle\cdot,\cdot\rangle$ and the induced Hermitian connection.
Moreover, one has the holomorphic bundle isomorphism $T^{(1,0)}\mathbb{CP}^n\cong \hom(L,L^\perp)$. We endow all these bundles with Hermitian connections induced from the Hermitian connection on the trivial $\mathbb{C}$-bundle. Then by Koszul-Malgrange theorem all line subbundles of $M\times\mathbb{C}^{n+1}$ are automatically holomorphic. 
The composition of the complexified differential of $\psi$ with the projection onto $T^{(1,0)}\mathbb{CP}^n$ yields the map
$d^{\mathbb{C}}\psi\in\hom(T^{\mathbb{C}}M\otimes\psi^*L,\psi^*L^\perp)$. Taking the $(1,0)$-part of $d^\mathbb{C}\psi$ defines 
\begin{equation}
\label{partial}
\partial\colon T^{(1,0)}M\otimes\psi^*L\to\psi^*L^\perp.
\end{equation}
Similarly, the $(0,1)$-part of $d^\mathbb{C}\psi$ defines 
\begin{equation}
\label{barpartial}
\bar\partial\colon  T^{(0,1)}M\otimes\psi^*L\to\psi^*L^\perp
\end{equation}

Let $g$ be any metric compatible with the complex structure on $M$. Assume that $\psi\colon (M,g)\to\mathbb{CP}^n$ is a linearly full (i.e. its image is not contained in a proper projective subspace) harmonic map, where $\mathbb{CP}^n$ is endowed with the Fubini-Study metric. In local complex coordinates the harmonicity can be expressed as 
$$
(\nabla d\psi)\left(\frac{\partial}{\partial\bar z},\frac{\partial}{\partial z}\right)=(\nabla d\psi)\left(\frac{\partial}{\partial z},\frac{\partial}{\partial\bar z}\right)=0,
$$
 which is equivalent to the fact that $\partial$ ($\bar\partial$) defined in~\eqref{partial} (in~\eqref{barpartial}) is a(n) (anti-)holomorphic morphism of bundles. Thus their images can be defined across zeroes of $d\psi$ and give rise to a line subbundle $L_1$ ($L_{-1}$). Denoting $\psi^*L$ by $L_0$ we have a holomorphic map
$$
\partial_0\colon T^{(1,0)}M\otimes L_0\to L_1
$$
and an antiholomorphic map
$$
\bar\partial_0\colon T^{(0,1)}M\otimes L_0\to L_{-1}
$$

Bundles $L_1$ and $L_{-1}$ correspond to maps $\psi_1,\psi_{-1}\colon M\to\mathbb{CP}^n$. It is proved in~\cite{CW1} that if $\psi_0=\psi$ is harmonic then so are $\psi_1$ and $\psi_{-1}$. Repeating the process one constructs a sequence of bundles $\{L_p\}$, holomorphic maps
$$
\partial_p\colon T^{(1,0)}M\otimes L_p\to L_{p+1}
$$
and antiholomorphic maps
$$
\bar\partial_p\colon T^{(0,1)}M\otimes L_p\to L_{p-1}.
$$
This collection of data is referred to as {\em a harmonic sequence associated to $\psi = \psi_0$}. 

The map $\partial_p$ is a holomorphic section of $(T^{(1,0)}M)^*\otimes L_p^*\otimes L_{p+1}$ and therefore one has
\begin{equation}
\label{ramification}
0 \leqslant c_1((T^{(1,0)}M)^*\otimes L_p^*\otimes L_{p+1}) = 2\gamma-2+c_1(L_{p+1}) - c_1(L_p),
\end{equation}
where $\gamma$ is the genus of $M$.

If $\partial_p\equiv0$ ($\bar\partial_p\equiv0$) but $\partial_{p-1}\not\equiv0$ ($\bar\partial_{p-1}\not\equiv 0$), then we say that the harmonic sequence terminates with $L_p$ at the right (left). In this case the map $\psi_p$ is antiholomorphic (holomorphic) and the harmonic sequence coincides with its Frenet frame.

\begin{remark}
\label{orientation:remark}
The harmonic sequence associated to $\psi$ depends on the choice of the orientation. Let $\overline{\psi}$ denote the same map $\psi$ considered as a map from $\overline M$, the surface with the same conformal class of metrics, but with the opposite orientation. If $z$ is a local holomorphic coordinate on $M$, then $\bar z$ is a local holomorphic coordinate on $\overline{M}$. As a result, the roles of $\partial_p$ and $\bar\partial_p$ are reversed. Thus, if $\{L^\psi_i\}$ is a harmonic sequence associated to $\psi$ and $\{L_i^{\bar\psi}\}$ is a harmonic sequence associated to $\overline\psi$, then $L_{-i}^\psi = L_i^{\bar\psi}$.
\end{remark}

Let $\pi$ be a projection $\pi\colon\mathbb{S}^n\to\mathbb{RP}^n$ and $i$ be an embedding $i\colon\mathbb{RP}^n\to\mathbb{CP}^n$. Since $i$ is totally geodesic, for any harmonic map $\Phi\colon(M,g)\to\mathbb{S}^n$ the composition $\psi=i\circ\pi\circ\Phi$ is harmonic. Moreover, $\Phi$ is linearly full iff $\psi$ is linearly full. 
For the remainder of this section we assume $M=\mathbb{S}^2$. Therefore, one can omit the metric $g$ from the notations and let $n=2m$. Let $\{L_i\}$ be the harmonic sequence associated to $\psi$. 
We remark the following properties.
\begin{itemize}
\item[1)] One has $\langle\Phi,\Phi\rangle = 1$. Therefore, $\langle\partial_z\Phi,\Phi\rangle = \langle\partial_{\bar z}\Phi,\Phi\rangle = 0$, i.e. $\Phi$ is parallel.
\item[2)] $\Phi\colon M\to\mathbb{S}^n\subset\mathbb{R}^{n+1}\subset\mathbb{C}^{n+1}$ is a global nowhere zero section of $L_0$. Since $\Phi$ is parallel, $L_0$ is trivial and $c_1(L_0)=0$.
\item[3)] $\bar L_p = L_{-p}$.
\item[4)] If $\Phi$ is linearly full, then the harmonic sequence always terminates with $L_{-m}$ at the left and with $L_{m}$ at the right, see e.g.~\cite{Barbosa, Calabi}. Note that Barbosa does not use the language of harmonic sequences, but his maps $G_i$ are exactly local holomorphic sections of $L_{-i}$. The map $\psi_{-m}$ associated to $L_{-m}$ is a holomorphic curve and is called {\em the directrix}  of $\Phi$. 
\item[5)] If $\Phi$ is linearly full, the trivial bundle $\mathbb{S}^2\times \mathbb{C}^{2m+1}$ is a direct sum of all the elements in harmonic sequence, i.e. 
$$
\mathbb{S}^2\times \mathbb{C}^{2m+1} =\bigoplus_{i=-m}^m L_i.
$$ 
\end{itemize}

\subsection{Twistor correspondence} 

Let us denote by $(\cdot,\cdot)$ the $\mathbb{C}$-bilinear extension of the usual Euclidean inner product on $\mathbb{R}^{2m+1}$ to $\mathbb{C}^{2m+1}$. A $\mathbb{C}$-linear subspace $V\subset\mathbb{C}^{2m+1}$ is called {\em isotropic} if $(\cdot,\cdot)|_V\equiv 0$ or, equivalently, if $V\perp\bar V$. The twistor space $\mathcal Z_m$ of $\mathbb{S}^{2m}$ is defined to be the space of all $m$-dimensional isotropic subspaces of $\mathbb{C}^{2m+1}$ considered as a complex submanifold of the Grassman manifold $Gr_{m,2m+1}(\mathbb{C})$. 
If $L\subset \mathcal Z_m\times\mathbb{C}^{2m+1}$ is the tautological bundle over $\mathcal Z_m$, the holomorphic tangent bundle $T^{(1,0)}\mathcal Z_m$ is isomorphic to the subbundle $\hom^s(L,L^\perp)\subset \hom(L,L^\perp)$ consisting of morphisms skew symmetric with respect to $(\cdot,\cdot)$.

The twistor projection $\pi_m\colon \mathcal Z_m\to\mathbb{S}^{2m}$ sends $V$ to the unit normal to $V\oplus \bar V$ 
(the direction of the normal is dictated by a choice of the orientation on $\mathbb{S}^{2m}$). The projection $\pi_m$ is a Riemannian submersion and induces a decomposition of $T^{(1,0)}\mathcal Z_m$ into vertical and horizontal distributions.
The vertical distribution $\mathcal{VZ}_m$ is the kernel of the differential $d\pi_m$. The horizontal distribution is the orthogonal complement $\mathcal{HZ}_m = (\mathcal{VZ}_m)^\perp$. If $L_0 = (L\oplus\bar L)^\perp\subset \mathcal Z_m\times\mathbb{C}^{2m+1}$, then under the isomorphism $T^{(1,0)}\mathcal Z_m\cong\hom^s(L,L^\perp)$ one has $\mathcal{HZ}_m\cong \hom(L,L_0)$ and $\mathcal{VZ}_m\cong\hom^s(L,\bar L)$.
A holomorphic map $\Psi\colon \mathbb{S}^2\to\mathcal Z_m$ is called {\em horizontal} if the image of the $(1,0)$-part of the differential $d\Psi\colon T^{(1,0)}M\to T^{(1,0)}\mathcal Z_m$ lies in the horizontal distribution $\mathcal {HZ}_m$.

\begin{theorem}[Twistor correspondence, Barbosa~\cite{Barbosa}]

One has the following,
\begin{itemize}
\item[1)] If $\Psi\colon \mathbb{S}^2\to\mathcal Z_m$ is a horizontal holomorphic map, then $\pm\pi_m\circ\Psi$ are harmonic maps. All harmonic maps $\mathbb{S}^2\to\mathbb{S}^{2m}$ can be obtained in this way.
\item[2)] Let $\Phi\colon\mathbb{S}^2\to\mathcal \mathbb{S}^{2m}$ be a linearly full harmonic map and let $\{L_i\}$ be the corresponding harmonic sequence. Set $L_{<0} = \bigoplus_{i=-m}^{-1}L_i$. The map $\Psi\colon \mathbb{S}^2\to\mathcal Z_m$ given by $z\mapsto L_{<0}(z)$ is the only horizontal holomorphic map satisfying $\pi_m\circ\Psi = \pm\Phi$. 
The map $\Psi$ is referred to as the {\em twistor lift} of $\Phi$.
\end{itemize}
\end{theorem}

\subsection{Jacobi operator}
\label{Jacobi:sec}
 In this section we discuss the relationship between Jacobi operator and harmonic sequence. This results in a proof of Theorem~\ref{RP2index:thm}. Let $\Phi\colon\mathbb{S}^2\to\mathbb{S}^{2m}$ be a linearly full harmonic map. Let $\{L_i\}$ be the harmonic sequence associated to $\Phi$ and set $L_{<0} = \bigoplus_{i=-m}^{-1}L_i$; $L_{>0} = \overline {L_{<0}} = \bigoplus_{i=1}^{m}L_i$.
Let $V$ be a vector field in $\dom \,Q_{\Phi, E}$, i.e. $V\perp L_0$. Decompose $V = V_++V_-$ into $L_{>0}$ and $L_{<0}$-parts. Since $V$ is real and $L_{>0} = \overline {L_{<0}}$, one has $V_+ = \overline {V_-}$.
\begin{lemma}
\label{Jacobi:lemma}
Extend the Jacobi operator $J_g$ to $\mathbb{C}$-valued vector fields by $\mathbb{C}$-linearity. Then $L_{>0}$ and $L_{<0}$ are $J_g$-invariant. Furthermore, the following are equivalent,
\begin{itemize}
\item[1)] $J_gV = \lambda V$;
\item[2)] $J_gV_+ = \lambda V_+$;
\item[3)] $J_gV_- = \lambda V_-$ .
\end{itemize}
\end{lemma}
\begin{proof}
Let $\pi^\perp_0$ be an orthogonal projection onto $L_{>0}\oplus L_{<0} = L^\perp_0$. Recall that 
$$
J_g V = \pi^\perp_0\left(\Delta_g V - |\nabla \widetilde \Phi|^2_gV\right).
$$
For any local complex coordinate $z$, let $g=e^{2\omega}dzd\bar z$. Then one has
$$
J_g V = \frac{4}{e^{2\omega}}\pi^\perp_0\left(-\frac{\partial^2 V}{\partial z\partial \bar z} - \left(\frac{\partial \widetilde \Phi}{\partial z},\frac{\partial \widetilde \Phi}{\partial \bar z} \right)V\right).
$$
Our main observation is that the operator $\frac{\partial^2}{\partial z\partial \bar z}$ maps $L_{<0}$ to $L_{<0}\oplus L_0$ and $L_{>0}$ to $L_{>0}\oplus L_0$. In other words $\pi^\perp_0\frac{\partial^2}{\partial z\partial \bar z}$ leaves the spaces $L_{>0}$ and $L_{<0}$ invariant. Thus, these spaces are $J_g$-invariant and the statements $1)\implies 2)$ and $1)\implies 3)$ follow.

At the same time $2) \iff 3)$ since $J_g$ is real and $V_+ = \overline {V_-}$. Assuming either $2)$ or $3)$ one has
$$
J_gV = J_gV_-+J_gV_+ = \lambda V_-+\lambda V_+ = \lambda V.
$$
\end{proof}

Let $V$ be a real vector field in $\dom Q_{\Phi,E}$ so that $V=2\Re V_+$. We define the {\em conjugate} vector field $V^*$ to be $V^* = 2\Im V_+$, where $\Re$ and $\Im$ denote the real and imaginary part respectively.

\begin{lemma}
\label{conjugate:lemma}
The operation of taking conjugate vector field satisfies the following properties,
\begin{itemize}
\item[1)] $(J_gV)^* = J_g(V^*)$.
\item[2)] $(V^*)^* = -V$.
\end{itemize}
\end{lemma} 
\begin{proof}
By Lemma~\ref{Jacobi:lemma} the decomposition of $J_gV$ into $L_{>0}$ and $L_{<0}$-parts is $J_gV = J_gV_+ + J_gV_-$. Therefore, one has
$$
(J_gV)^* = -i(J_gV_+ - J_gV_-)  = J_g(-i(V_+-V_-))= J_g(V^*). 
$$

If $V^* = 2\Im V_+$, then $V^* = 2\Re(-iV_+)$ and therefore one has
$$
(V^*)^* = 2\Im(-iV_+) = -2\Re(V_+) = -V.
$$
\end{proof}
Lemma~\ref{conjugate:lemma} implies that the conjugation is $1$-to-$1$ linear map that preserves the eigenspaces of $J_g$. Furthermore, it does not have real eigenvectors and, therefore, both $\ind_E(\Phi)$ and $\nul_E(\Phi)$ are even. 

Finally, we are ready to prove Theorem~\ref{RP2index:thm}.
Let $\Psi$ be a linearly full harmonic map $\Psi\colon \mathbb{RP}^2\to\mathbb{S}^{2m}$ and let $\Phi\colon\mathbb{S}^2\to\mathbb{S}^{2m}$ be its antipodal lift. 
Set $\Phi^\sigma = \Phi\circ\sigma$, where $\sigma$ is the antipodal involution on $\mathbb{S}^2$. Since $\sigma$ inverses orientation and $\Phi$ is a lift of $\Psi$, one has that $\Phi^\sigma$ is the same map as $\Phi$, but with the orientation of $\mathbb{S}^2$ reversed. Using the notations of Remark~\ref{orientation:remark} one has that $\Phi^\sigma = \overline\Phi$. Therefore, $L^\Phi_{>0} = L^{\Phi^\sigma}_{<0}$ and $L^\Phi_{<0} = L^{\Phi^\sigma}_{>0}$. As a result, one has $\sigma^*L^\Phi_{>0} = L^\Phi_{<0}$ and $\sigma^*L^\Phi_{<0} = L^\Phi_{>0}$.

The involution $\sigma$ induces an isometric involution $\sigma^*$ on $\dom Q_{E,\Phi}$. Since it is an isometry, it commutes with $J_{g,\Phi}$, therefore, it preserves a decomposition of $\dom Q_{E,\Phi}$ into $\sigma$-odd and $\sigma$-even fields. The eigenvalues of $J_{g,\Psi}$ coincide with the eigenvalues of $J_{g,\Phi}$ restricted to the space of $\sigma$-even fields. Let $V$ be a $\sigma$-even field. We claim that the conjugate $V^*$ is a $\sigma$-odd field. Indeed,
\begin{equation*}
\begin{split}
\sigma^* (V^*) &= 2\sigma^* \Im V_+ = -2i\sigma^*(V_+-V-) \\
&= 2i (\sigma^*V_--\sigma^*V_+ )=-2\Im (\sigma^* V)= -(\sigma^* V)^*,
\end{split}
\end{equation*}
 where we used that the application of $\sigma^*$ interchanges $L^\Phi_{<0}$ and $L^\Phi_{>0}$. Similarly, the conjugate of an odd field is an even field. Thus, for each eigenvalue $\lambda$ of $J_{g,\Phi}$ exactly half of the corresponding eigenfunctions are even. Applying it to negative and to zero eigenvalues implies Theorem~\ref{RP2index:thm}.

\subsection{Jacobi fields}
\label{jacobi:sec}
The purpose of this section is to define the first order analog of twistor correspondence, i.e. to lift any Jacobi field to a vector field on the twistor space $\mathcal Z_m$ satisfying additional properties. The first order correspondence has been established in~\cite{LW} for $m=2$. The main difficulty is the presence of branch points. It has been overcome in~\cite{LW} for $m=2$ with the help of an isomorphism $\mathcal Z_2\cong \mathbb{CP}^3$. The authors indicate that their methods are specific to $m=2$. Below we propose a different approach using harmonic sequence that allows us to extend the $1$-st order twistor correspondence to an arbitrary value of $m$. Finally, we remark that in the most general context the $1$-st order twistor correspondence was studied in~\cite{Simoes}. However, the author considers local lifts away from branch points, so we can not use their results directly.

In the remainder of the paper we only work with $\mathbb{S}^2$, so we once and for all fix the orientation so that we can always speak of the corresponding harmonic sequence.
Let $\Phi$ be a linearly full harmonic map $\Phi\colon\mathbb{S}^2\to\mathbb{S}^{2m}$ and $\{L_i\}$ be the corresponding harmonic sequence. We denote by $\Sigma_s\subset \mathbb{S}^2$ the set of higher singularities of $\Phi$, i.e. the discrete set of zeroes of all maps $\partial_i$ and $\bar\partial_i$. Let $\Sigma' = \mathbb{S}^2\setminus \Sigma_s$ and $f_0 = \Phi$. Thus, on $\Sigma'$ these operators are invertible. Then by~\cite{BJRW} for any local complex coordinate $z$ one can choose local nowhere zero holomorphic sections $f_p$ of $L_p$ such that the following formulae hold,

\begin{equation}
\label{harmseq:eq}
\begin{split}
&\frac{\partial f_p}{\partial z} = f_{p+1} + \left(\frac{\partial}{\partial z}\ln |f_p|^2\right)f_p; \\
&\frac{\partial f_p}{\partial \bar z} = -\gamma_{p-1}f_{p-1}; \\
&\gamma_p = \frac{|f_{p+1}|^2}{|f_p|^2}; \\
&\frac{\partial^2 }{\partial z\partial \bar z} \ln |f_p|^2 = \gamma_p-\gamma_{p-1}; \\
&\frac{\partial^2 }{\partial z\partial \bar z} \ln \gamma_p = \gamma_{p+1}-2\gamma_p+\gamma_{p-1}.
\end{split}
\end{equation}
Note that $\gamma_0=\gamma_{-1} = \left(\frac{\partial \Phi}{\partial z},\frac{\partial \Phi}{\partial \bar z} \right)$
and $\frac{1}{|f_{-p}|^2} = \gamma_{-p}\ldots\gamma_{-1}$. 

Let $V=V_0$ be a Jacobi field along $\Phi$. Assume that there is a one parameter family of harmonic maps $\Phi_t$, such that $V = \frac{d}{dt}|_{t=0}\Phi_t$, i.e. $V$ is integrable. Then one has the corresponding family of local sections $(f_t)_p$. Setting $V_p = \frac{d}{dt}|_{t=0}(f_t)_p$ and taking the $t$ derivatives of~\eqref{harmseq:eq} yields a series of recursive formulae relating $V_p$ and $f_p$ for various values of $p$. In general, it is unknown whether all Jacobi fields are integrable. Nevertheless, one can use these recursive formulae to define $V_p$ starting from $\{f_p\}$ and $V_0$. This is a motivation for the definitions below.
%
%
%
%

Let $\hat\gamma_0 = \frac{1}{4}\pi_{\underline \Phi}(\Delta V)$. Set $\hat\gamma_{-1} = \hat\gamma_0$. For $p\geqslant 0$ define inductively 
\begin{equation}
\label{gammahat:def}
\hat\gamma_{-p-1} = \frac{\partial^2 }{\partial z\partial \bar z}\left(\frac{\hat\gamma_{-p}}{\gamma_{-p}}\right) + 2\hat\gamma_{-p} - 
\hat\gamma_{-p+1}
\end{equation}
and 
$$
V_{-p-1} = -\frac{1}{\gamma_{-p-1}}\left(\frac{\partial V_{-p}}{\partial \bar z} + \hat\gamma_{-p-1}f_{-p-1}\right).
$$

\begin{proposition}
\label{dzV}
One has 
$$
\frac{\partial V_{-p}}{\partial z} = V_{-p+1} + \left(\frac{\partial}{\partial z}\ln |f_{-p}|^2\right)V_{-p} -\frac{\partial}{\partial z}
\left(\sum_{i=-p}^{-1}\frac{\hat\gamma_{i}}{\gamma_{i}}\right)f_{-p}
$$
\end{proposition}
\begin{proof}
The proof is by induction.

Base of the induction: $p=1$. Recall that $V_0$ is a Jacobi field and, therefore, by definition of $\hat\gamma_{-1}$ one has
$$
\frac{\partial^2 V_{0}}{\partial \bar z\partial z} = -\gamma_{-1}V_0-\hat\gamma_{-1}f_0
$$
Thus, by definition of $V_{-1}$ one has 
\begin{equation*}
\begin{split}
&\frac{\partial V_{-1}}{\partial z} = \frac{\partial }{\partial z}\left(-\frac{1}{\gamma_{-1}}\right)\left(\frac{\partial V_{0}}{\partial \bar z} + \hat\gamma_{-1}f_{-1}\right) \\
&- \frac{1}{\gamma_{-1}}\left(\frac{\partial^2 V_{0}}{\partial \bar z\partial z} + \frac{\partial \hat\gamma_{-1}}{\partial z}f_{-1} + \hat\gamma_{-1}\left(f_0 + \left(\frac{\partial }{\partial z}\ln|f_{-1}|^2\right) f_{-1}\right)\right) = \\
&\frac{\partial }{\partial z}\left(\frac{1}{\gamma_{-1}}\right)\gamma_{-1} V_{-1} - \frac{1}{\gamma_{-1}}\left(-\gamma_{-1} V_0 +\left( \frac{\partial \hat\gamma_{-1}}{\partial z} + \hat\gamma_{-1}\frac{\partial }{\partial z}\ln|f_{-1}|^2\right)f_{-1}\right).
\end{split}
\end{equation*}
Using that $\gamma_{-1} = |f_{-1}|^{-2}$ completes the proof of the base.

Step of the induction: assume the formula is proved for $p$, we prove it for $p+1$. 
First remark that by the step of induction 
$$
\frac{\partial^2}{\partial z\partial \bar z} V_{-p} = \frac{\partial }{\partial \bar z}\left(V_{-p+1} + \left(\frac{\partial}{\partial z}\ln |f_{-p}|^2\right)V_{-p} -\frac{\partial}{\partial z}
\left(\sum_{i=-p}^{-1}\frac{\hat\gamma_{i}}{\gamma_{i}}\right)f_{-p} \right).
$$
After applying the definition of $V_{-p-1}$ to the expression $\frac{\partial V_{-p-1}}{\partial z}$ and using the above formula, we obtain an expression involving the vectors $V_{-p}$, $V_{-p-1}$, $f_{-p}$ and $f_{-p-1}$ with some coefficients. The coefficients are as follows.

Before $V_{-p}$,
$$
-\frac{1}{\gamma_{-p-1}}\left(-\gamma_{-p} + \frac{\partial^2}{\partial z\partial \bar z}\ln|f_{-p}|^2\right) = -\frac{1}{\gamma_{-p-1}}\left(-\gamma_{-p} + \gamma_{-p}-\gamma_{-p-1}\right) = 1.
$$

Before $V_{-p-1}$,
$$
\frac{\partial}{\partial z}\left(\frac{1}{\gamma_{-p-1}}\right)\gamma_{-p-1} + \frac{\partial}{\partial z}\ln|f_{-p}|^2 = \frac{\partial}{\partial z}\left(-\ln\gamma_{-p-1} + \ln|f_{-p}|^2\right) = \frac{\partial}{\partial z}\ln|f_{-p-1}|^2.
$$

Before $f_{-p}$,
\begin{equation*}
\begin{split}
-\frac{1}{\gamma_{-p-1}}\left(\hat\gamma_{-p-1} - \frac{\partial^2}{\partial z\partial \bar z}\left(\sum_{i=-p}^{-1}\frac{\hat\gamma_i}{\gamma_i}\right)-\hat\gamma_{-p}\right)=0,
\end{split}
\end{equation*}
where we have used that by definition of $\hat\gamma_i$ one has
\begin{equation*}
\begin{split}
\frac{\partial^2}{\partial z\partial \bar z}\left(\sum_{i=-p}^{-1}\frac{\hat\gamma_i}{\gamma_i}\right) &= \sum_{i=-p}^{-1}\left(\hat\gamma_{i-1}-2\hat\gamma_i+\hat\gamma_{i+1}\right) \\
&= \hat\gamma_{-p-1}-\hat\gamma_{-p}-\hat\gamma_{-1}+\hat\gamma_{0} = \hat\gamma_{-p-1}-\hat\gamma_{-p}.
\end{split}
\end{equation*}

Before $f_{-p-1}$,
\begin{equation*}
\begin{split}
&-\frac{1}{\gamma_{-p-1}}\left(\frac{\partial \hat\gamma_{-p-1}}{\partial z} + \hat\gamma_{-p-1}\frac{\partial}{\partial z}\ln|f_{-p-1}|^2 - \hat\gamma_{-p-1}\frac{\partial}{\partial z}\ln|f_p|^2 \right.  \\
&\left.+ \gamma_{-p-1}\frac{\partial}{\partial z}\left(\sum_{i=-p}^{-1}\frac{\hat\gamma_i}{\gamma_i}\right)\right)=- \frac{\partial}{\partial z}\left(\sum_{i=-p}^{-1}\frac{\hat\gamma_i}{\gamma_i}\right) \\
&- \frac{1}{\gamma_{-p-1}}\left(\frac{\partial \hat\gamma_{-p-1}}{\partial z} - \hat\gamma_{-p-1}\frac{\partial}{\partial z}\ln\gamma_{-p-1}\right) = - \frac{\partial}{\partial z}\left(\sum_{i=-p-1}^{-1}\frac{\hat\gamma_i}{\gamma_i}\right),
\end{split}
\end{equation*}
which completes the proof.
\end{proof}

Set $L_{\geqslant 0}=\bigoplus_{p=0}^m L_p$, $L_{\leqslant 0}=\bigoplus_{p=0}^m L_{-p}$, $L_{> 0}=\bigoplus_{p=1}^m L_p$, $L_{< 0}=\bigoplus_{p=1}^m L_{-p}$. We let $\pi_{\geqslant 0}$, $\pi_{\leqslant 0}$, $\pi_{>0}$, $\pi_{<0}$ be the corresponding orthogonal projections.


Recall that to each linearly full harmonic map $\Phi\colon\mathbb{S}^2\to\mathbb{S}^{2m}$ we associated a holomorphic twistor lift $\Psi\colon\mathbb{S}^2\to\mathcal Z_m$. Since $\Psi^*T^{(1,0)}\mathcal Z_m\cong \hom^s(L_{<0},L_{\geqslant 0})$, we need to construct an element of $\hom^s(L_{<0},L_{\geqslant 0})$ from the Jacobi field $V_0$. This motivates the following definitions. We set $l_{-p}\in\hom(L_{-p},L_{\geqslant 0})$, $p=0,\ldots, m$ and $l\in\hom(L_{<0},L_{\geqslant 0})$ by setting locally on $\Sigma'$ that $l_{-p}(f_{-p}) = \pi_{\geqslant 0} V_{-p}$ for $p\geqslant 0$ and $l(f_{-p}) = l_{-p}(f_{-p})$ for $p\geqslant 1$. Our next goal is to extend $l$ across singular points and to show that $l\in\hom^s(L_{<0},L_{\geqslant 0})$

In the following we use the notation $\partial_{z,p}$ for the map $\partial_p\left(\frac{\partial}{\partial z},\cdot\right)\colon L_p\to L_{p+1}$. Similarly, we use $\bar\partial_{\bar z,p}$ for the map $\bar\partial_p\left(\frac{\partial}{\partial \bar z},\cdot\right)\colon L_p\to L_{p-1}$.

\begin{proposition}
\label{antisymmetric:prop}
For all $p\geqslant 0$ one has the following, 
\begin{equation}
\label{invariantl}
\nabla_{\bar z} l_{-p}= l_{-p-1}\circ \bar\partial_{\bar z,-p},
\end{equation}
where $\nabla$ is the connection in $\hom(L_{-p},L_{\geqslant 0})$. In particular, $l_{-p}$ and $l$ are independent of the choice of a local complex coordinate.

Moreover,
$$
\nabla_{\bar z} l = 0,
$$
where $\nabla$ is a connection in $\hom(L_{<0},L_{\geqslant 0})$, i.e. $l$ is holomorphic on $\Sigma'$.
\end{proposition}
\begin{proof}
Both sides of~\eqref{invariantl} are linear maps, therefore, it is sufficient to check the relation for a specific section. Let us substitute $f_{-p}$ and check that the equality is satisfied. 

Note that $\frac{\partial}{\partial\bar z} L_p\subset L_{p-1}\oplus L_p$ and therefore, $\pi_{\geqslant 0}\frac{\partial}{\partial\bar z}\pi_{\geqslant 0} = \pi_{\geqslant 0}\frac{\partial}{\partial\bar z}$. Thus,
\begin{equation*}
\begin{split}
(\nabla_{\bar z} l_{-p})(f_{-p}) &= \pi_{\geqslant 0}\frac{\partial}{\partial\bar z}\left(\pi_{\geqslant 0} V_{-p}\right) - l_{-p}\left(\pi_{-p} \frac{\partial}{\partial\bar z} f_{-p}\right) = \pi_{\geqslant 0}\frac{\partial}{\partial\bar z}V_{-p} \\ &= \pi_{\geqslant 0}(-\gamma_{-p-1}V_{-p-1} - \hat\gamma_{-p-1}f_{-p-1}) = \pi_{\geqslant 0}(-\gamma_{-p-1}V_{-p-1}) \\ &= l_{-p-1} (-\gamma_{-p-1}f_{-p-1}) = l_{-p-1}(\bar\partial_{\bar z,-p}(f_{-p})).
\end{split}
\end{equation*}
Therefore, one has the following expression,
$$
l_{-p-1} = (\nabla_{\bar z} l_{-p})\circ (\bar\partial_{\bar z,-p})^{-1},
$$
which allows one to show the independence of the choice of coordinates by induction on $p$.

To prove the second equality, we once again substitute $f_{-p}$ and check that the equality holds. Similarly to the previous computation, one has
\begin{equation*}
\begin{split}
(\nabla_{\bar z} l)(f_{-p}) &= \pi_{\geqslant 0}\frac{\partial}{\partial\bar z}(\pi_{\geqslant 0} V_{-p}) - l\left(\frac{\partial}{\partial\bar z} f_{-p}\right) = \pi_{\geqslant 0}\frac{\partial}{\partial\bar z}V_{-p} - l(-\gamma_{-p-1}f_{-p-1}) \\
&= -\pi_{\geqslant 0}(\gamma_{-p-1} V_{-p-1} - \gamma_{-p-1} V_{-p-1}) = 0.
\end{split}
\end{equation*}
\end{proof}

The equation~\eqref{invariantl} sheds light on the behaviour of $l$ in the neighbourhood of higher singularities. Recall that $\bar\partial_{\bar z,-p}\colon L_{-p}\to L_{-p-1}$ is an anti-holomorhic map.
\begin{proposition}
There exists a non-negative integer $K$, such that for any point $x\in\Sigma_s$ and any local holomorphic coordinate $z$ with $z(x)=0$, one has that $\bar z^K l$ can be extended smoothly across $x$.
\end{proposition}
\begin{proof}
Since $L_{<0}=\bigoplus_{i=1}^m L_{-p}$ it is sufficient to show the existence of $K$ for each $l_{-p}$, $p =0,\ldots, m$. We prove the assertion by induction on $p$. The base is $p=0$. Indeed, $l_0$ is defined on $\mathbb{S}^2$, since $f_0$ is nowhere zero section and $V_0$ is defined everywhere on $\mathbb{S}^2$.

Suppose that the assertion is proved for $p$. Let $s_{-p-1}$ be a local anti-holomorphic section of $L_{-p-1}$ in the neighbourhood of $x$ such that  $s_{-p-1}(x)\ne 0$. If $k_x$ is the ramification order of $\bar\partial_{-p}$ at $x$, then there exists a local holomorphic coordinate $z$ with $z(x)=0$ and a local anti-holomorphic section $s_{-p}$ of $L_{-p}$ such that $s_{-p}(x)\ne 0$ and $\bar\partial_{\bar z,-p}(s_{-p}) = \bar z^{k_x}s_{-p-1}$. Then by formula~\eqref{invariantl} one has
\begin{equation*}
\begin{split}
(\nabla_{\bar z} (\bar z^{K+1}l_{-p}))(s_{-p}) &= (K+1)\bar z^{K} l_{-p}(s_{-p}) + \bar z^{K+1}l_{-p-1} (\bar\partial_{\bar z,-p}(s_{-p})) \\&= (K+1)\bar z^{K} l_{-p}(s_{-p}) + \bar z^{K+1+k_x}l_{-p-1}(s_{-p-1}).
\end{split}
\end{equation*}
Since other terms in the equality are smoothly defined across $x$, the section $\bar z^{K+1+k_x}l_{-p-1}(s_{-p-1})$ can be smoothly extended to $x$. As $s_{-p-1}(x)\ne 0$ and $l_{-p-1}$ is linear, it follows that $\bar z^{K+1+k_x}l_{-p-1}$ can be smoothly defined at $x$. Finally, since $\Sigma_s$ is discrete and finite, the numbers $k_x$ are uniformly bounded in $x$. Therefore, one can choose a possibly bigger $K'$ that satisfies the assertion of the proposition.
\end{proof}

\begin{proposition}
The section $l$ can be smoothly extended across $\Sigma_s$.
\end{proposition}
\begin{proof}
The statement is a consequence of two previous propositions. On one hand, $l$ is holomorhic, i.e. it can only have removable singularities, poles or essential singularities on $\Sigma_s$. On the other hand $\bar z^K l$ is smooth, so $l$ can not have either poles or essential singularities. Thus, all singularities are removable.
\end{proof}
\begin{remark}
This proof is reminiscent of~\cite[Lemma 2.9]{LW}.
\end{remark}

In the following proofs we will often check certain equalities on $\Sigma'$, where we can use the explicit expressions for local sections $f_{-p}$, and then conclude the equality at singular points by continuity.


\begin{proposition}
\label{1isotropic:prop}
For all $p\geqslant 0$ one has that $(l_{-p}(\cdot),\cdot)$ is a zero section of $(L^*_{-p})^2$.
\end{proposition}
\begin{proof}
The proof is by induction on $p$. The base $p=0$ is by definition since $V_0\perp f_0$.
For $p>0$ one has that $(l_{-p}(\cdot),\cdot) = (l(\cdot),\cdot)$, i.e. it is defined on $\mathbb{S}^2$ and it sufficient to 
check the equality for $f_{-p}$, i.e. that $ (f_{-p},V_{-p}) = 0. $

Suppose that $(f_{-p},V_{-p})=0$. Differentiating it with respect to $\bar z$ yields
$$
(-\gamma_{-p-1}f_{-p-1}, V_{-p}) + (f_{-p}, -\gamma_{-p-1}V_{-p-1}-\hat\gamma_{-p-1}f_{-p-1})=0.
$$
For all $p\geqslant 0$ one has $(f_{-p},f_{-p-1}) = 0$ and, thus,  
\begin{equation}
\label{eq1}
(f_{-p},V_{-p-1}) + (f_{-p-1},V_{-p})=0.
\end{equation}

Let $H(\cdot,\cdot) = (l_{-p-1}(\cdot),\cdot) = (l(\cdot),\cdot)$, we claim that $H$ is an anti-holomorphic section of $(L_{-p-1}^2)^*$. Indeed,
\begin{equation*}
\begin{split}
&(\nabla_z H)(f_{-p-1},f_{-p-1}) \\
&= \frac{\partial}{\partial z}(V_{-p-1},f_{-p-1}) -(l_{-p-1}(\nabla_z^{L_{-p-1}}f_{-p-1}),f_{-p-1}) - (l(f_{-p-1}), \nabla_z^{L_{-p-1}}f_{-p-1}) \\
&=(V_{-p} + \frac{\partial}{\partial z}\left(\ln|f_{-p-1}|^2\right) V_{-p},f_{-p-1}) +  (V_{-p-1}, f_{-p} + \frac{\partial}{\partial z}\left(\ln|f_{-p-1}|^2\right)f_{-p-1}) \\
&-2\frac{\partial}{\partial z}\left(\ln|f_{-p-1}|^2\right)(V_{-p-1},f_{-p-1}) = 0,
\end{split}
\end{equation*}
where in the last equality we used~\eqref{eq1}.

At the same time, according to~\eqref{ramification}, for $p\geqslant0$ one has $c_1(L_{-p-1}) \leqslant -2(p+1)<0$, i.e. $(L^*_{-p-1})^2$ does not have non-zero anti-holomorphic sections. Therefore $H\equiv 0$ and the proof is complete. 
\end{proof}

\begin{corollary}
\label{2isotropy:cor}
For all $p,q\geqslant 0$ and for all $v\in L_{-p}$, $w\in L_{-q}$ one has that
$$
(l_{-p}(v),w)+ (v, l_{-q}(w)) = 0.
$$
In particular, $l\in\hom^s(L_{<0},L_{\geqslant 0})$.
\end{corollary}
\begin{proof}
Let $H_{p,q}$ be a bilinear form on $L_{-p}\otimes L_{-q}$ defined by $H_{q,p}(v,w)=H_{p,q}(v,w)=(l_{-p}(v),w)+ (v, l_{-q}(w))$. The proposition asserts that $H_{p,q}$ are all identically zero. By Proposition~\ref{1isotropic:prop} for all $p\geqslant 0$ one has $H_{p,p}=0$. Moreover, by equality~\eqref{eq1} in the proof of Proposition~\ref{1isotropic:prop} one also has $H_{p,p+1}=0$. 

We prove the assertion by induction on $p+q$. If $p+q=0$, then $p=q=0$ and $H_{0,0}=0$, so the base is proved.
Assume that $H_{p,q}=0$ for all $p+q=N$.

For any such $p,q$ applying $\frac{\partial}{\partial \bar z}$ to $H_{p,q}(f_{-p},f_{-q})$ yields,
\begin{equation*}
\begin{split}
0 = -\gamma_{-p-1} H_{p+1,q}(f_{-p-1},f_{-q}) -\gamma_{-q-1} H_{p,q+1}(f_{-p},f_{-q-1}).
\end{split}
\end{equation*}  
Moreover, $\gamma_i\ne 0$ on $\Sigma'$. Thus, on $\Sigma'$ all functionals $H_{p',q'}$ with $p'+q'=N+1$ can be obtained from one another by a multiplication by a nowhere zero function. However, by the discussion at the beginning of the proof, at least one of these functionals is zero, and, therefore, all of them are.
\end{proof}

This proposition allows us to prove Proposition~\ref{infinitesimal isometry}.
\begin{proof}[Proof of Proposition~\ref{infinitesimal isometry}] 
It is sufficient to prove the statement for $M=\mathbb{S}^2$. Indeed, if $M=\mathbb{RP}^2$ one can lift $V$ to an even Jacobi field and apply the proposition to the lift.

The condition in the proposition is equivalent to $\hat\gamma_{-1}=\hat\gamma_0=0$. Thus, by~\eqref{gammahat:def}, $\hat\gamma_{-p} = 0$ for all $p$. 

For any complex local coordinate $z$ define a local section $A\in\hom(\underline{\mathbb{C}}^n,\underline{\mathbb{C}}^n)$ on $\Sigma'$ in the following way. On $L_{\leqslant 0}$ we set $Af_{-p} = V_{-p}$ and for $v\in L_{\geqslant 0}$ we set $Av = \overline{A\bar v}$. Since $V_0$ is real, the definition is consistent on $L_0$. Furthermore, since $\hat\gamma_{-p}=0$ we see that the sequence $\{V_{-p}\}$ satisfy the exact same differential equations as $\{f_{-p}\}$. As a result, the definition of $A$ does not depend on the choice of a local complex coordinate.

 We claim that $A$ is constant, i.e. $\frac{\partial}{\partial z} A=\frac{\partial}{\partial \bar z} A=0$. Indeed, if $p>0$, then
\begin{equation*}
\begin{split}
\left(\frac{\partial}{\partial z} A\right)(f_{-p})& = \frac{\partial}{\partial z} (A f_{-p}) - A\left(\frac{\partial}{\partial z} f_{-p}\right) \\
&= \frac{\partial}{\partial z} V_{-p} - A\left(f_{-p+1} +\frac{\partial}{\partial z}\ln|f_{-p}|^2f_{-p}\right) = 0,
\end{split}
\end{equation*}
by Proposition~\ref{dzV} since $\hat\gamma_{-p} = 0$ for all $p$. Similarly, if $q\geqslant 0$, then
\begin{equation*}
\begin{split}
\left(\frac{\partial}{\partial \bar z} A\right)(f_{-q})& = \frac{\partial}{\partial \bar z} (A f_{-q}) - A\left(\frac{\partial}{\partial \bar z} f_{-q}\right) \\
&= \frac{\partial}{\partial \bar z} V_{-q} - A(-\gamma_{-q-1}f_{-q-1}) = 0,
\end{split}
\end{equation*}
by definition of $V_{-q-1}$. Conjugating the two previous computations yields the claim that $A$ is constant.

Finally we prove that $A\in\mathfrak{so}(2m+1)$. Note that $Af_0\perp f_0$ by definition. Furthermore, for $p,q\geqslant 0$ one has $(Af_{-p},f_{-q}) = (l_{-p}(f_{-p}),f_{-q})$. Therefore by Corollary~\ref{2isotropy:cor} one has 
$(Af_{-p},f_{-q}) + (f_{-p},Af_{-q})=0$. Taking conjugate yields $(Af_{p},f_{q}) + (f_{p},Af_{q})$. Finally, we show that 
\begin{equation}
\label{eqA}
(Af_{-p},f_{q}) + (f_{-p},Af_{q})=0.
\end{equation}
First of all, we already proved~\eqref{eqA} for any $q$ and $p=0$. Thus, it is sufficient to show~\eqref{eqA} for $p\geqslant 1$.
We show it by induction on $q$. The base $q=0$ is already established. Suppose~\eqref{eqA} is proved for $q$. 
To show it for $q+1$ we apply $\frac{\partial}{\partial z}$ to both sides of~\eqref{eqA},
\begin{equation*}
\begin{split}
&\left(A\frac{\partial f_{-p}}{\partial z}, f_q\right) + \left(Af_{-p}, \frac{\partial f_q}{\partial z}\right) + \left(\frac{\partial f_{-p}}{\partial z}, Af_q\right) + \left(f_{-p}, A\frac{\partial f_{q}}{\partial z}\right) \\
&=\left(Af_{-p}, \frac{\partial f_q}{\partial z}\right) + \left(f_{-p}, A\frac{\partial f_{q}}{\partial z}\right) = (Af_{-p},f_{q+1}) + (f_{-p},Af_{q+1}),
\end{split}
\end{equation*}
where we used the step of induction twice. As a result, we have that $A$ is skew symmetric with respect to $(\cdot,\cdot)$ on $\mathbb{C}^{2m+1}$. Thus, its restriction to $\mathbb{R}^{2m+1}$ is an element of $\mathfrak{so}(2m+1)$.
\end{proof}

For any section $s\in \hom(L_{<0},L_{\geqslant 0})$ we define its vertical part $s^V$ and horizontal part $s^H$ to be   $s^H,s^V\in\hom(L_{<0},L_{\geqslant 0})$ such that $s^V=\pi_{>0}s$, $s^H = \pi_0 s$. This corresponds to taking vertical and horizontal parts in $T^{(1,0)}\mathcal Z_m$.

\begin{proposition}
\label{lala:prop}
One has
$$
(\nabla_z l)^V = l_0\circ\partial_{z,-1}\circ\pi_{-1}.
$$
\end{proposition}
\begin{proof}
We check the statement for $f_{-p}$ on $\Sigma'$ and then use continuity to conclude it on $\Sigma$. We treat cases $p=1$ and $p>1$ separately.

{\bf Case $p>1$}. In this case $\nabla^{L_{<0}}_zf_{-p} = \frac{\partial}{\partial z}f_{-p}$ and one has,
\begin{equation*}
\begin{split}
(\nabla_z l)^V(f_{-p}) = \pi_{>0}\frac{\partial}{\partial z}(\pi_{\geqslant 0} V_{-p}) - \pi_{>0}l\left(\frac{\partial}{\partial z}f_{-p}\right).
\end{split}
\end{equation*}
Furthermore, since $\frac{\partial}{\partial z} L_{-p}\subset L_{-p}\oplus L_{-p+1}$ one has that $ \pi_{>0}\frac{\partial}{\partial z}\pi_{\geqslant 0} =  \pi_{>0}\frac{\partial}{\partial z}$ and therefore,
\begin{equation*}
\begin{split}
(\nabla_z l)^V(f_{-p}) &= \pi_{>0}\left(V_{-p+1} + \left(\frac{\partial}{\partial z}\ln |f_{-p}|^2\right)V_{-p} -\frac{\partial}{\partial z}
\left(\sum_{i=-p}^{-1}\frac{\hat\gamma_{i}}{\gamma_{i}}\right)f_{-p}\right) \\ 
&-\pi_{>0}\left(V_{-p+1} + \left(\frac{\partial}{\partial z}\ln |f_{-p}|^2\right)V_{-p} -\frac{\partial}{\partial z}\right)=0.
\end{split}
\end{equation*}

{\bf Case $p=1$}. In this case $\nabla^{L_{<0}}_zf_{-1} = \left(\frac{\partial}{\partial z}\ln|f_{-1}|^2\right)f_{-1}$ and computations similar to the previous one yield
\begin{equation*}
\begin{split}
(\nabla_z l)^V(f_{-1}) &= \pi_{>0}\left(V_0 + \frac{\partial}{\partial z}\ln|f_{-1}|^2V_{-1}\right) - \pi_{>0}\left(\left(\frac{\partial}{\partial z}\ln|f_{-1}|^2\right)V_{-1}\right) \\ &= \pi_{>0}(V_0) = l_0(V_0) = l_0(\partial_{z,-1}(f_{-1})).
\end{split}
\end{equation*}
\end{proof}

Furthermore, let $l^*\in \hom(L_{\leqslant 0}, L_{>0})$ be the adjoint to $l$ with respect to the bilinear form $(\cdot,\cdot)$. 

\begin{proposition}
\label{adjoint:prop}
The adjoint $l^*$ satisfies the following properties
\begin{itemize}
\item[1)] $l^*|_{L_{<0}} = - l^V$;
\item[2)] $l^*|_{L_0} = -l_0$;
\item[3)] $V_0 =-2\Re(l^*(f_0))$;
\item[4)] $(\nabla_z l)^V = -l^*\circ\partial_{-1}(\partial_z)\circ\pi_{-1}$
\end{itemize}
\end{proposition}
\begin{proof}
Items $1)$ and $2)$ follow directly from Corollary~\ref{2isotropy:cor}.

To prove $3)$ we note that $V_0 = (\pi_{>0} + \pi_{<0}) V_0$. At the same time $L_{>0} = \overline{L_{<0}}$ and $V_0$ is real. Therefore, $\pi_{<0}V_0 =\overline{(\pi_{>0}V_0)}$ and 
$$
V_0 = \pi_{>0}(V_0) + \overline{(\pi_{>0}V_0)} = 2\Re (\pi_{>0}(V_0)) = -2\Re(l^*(f_0)).
$$

Item $4)$ follows from Proposition~\ref{lala:prop} and item $2)$.
\end{proof}

\subsection{Twistor fields}
\label{1twistor:sec}
Motivated by the contents of the previous section we propose the following definition.
\begin{definition}
Given a linearly full harmonic map $\Phi\colon\mathbb{S}^2\to\mathbb{S}^{2m}$ with a twistor lift $\Psi\colon\mathbb{S}^2\to \mathcal Z_m$ we say that a section $l$ of $\Psi^*T^{(1,0)}\mathcal Z_m$ (or, equivalently, a section of $\hom^s (L_{<0},L_{\geqslant 0})$) is a {\em twistor field} along $\Psi$ if,
\begin{itemize}
\item[1)] $l$ is holomorphic, and
\item[2)] $(\nabla_z l)^V = -l^*\circ\partial_{z,-1}\circ\pi_{-1}$.
\end{itemize}
\end{definition}
\begin{remark}
Our concept of twistor field coincides with the concept of {\em infinitesimal holomorphic horizontal deformation} defined in~\cite{LW}, see the discussion after Proposition~\ref{twistoradjoint:prop}.
\end{remark}

Let us summarize the contents of the previous section using the language of twistor fields.
\begin{proposition}
Let $\Phi$ be linearly full harmonic map $\Phi\colon\mathbb{S}^2\to\mathbb{S}^{2m}$ with a twistor lift $\Psi\colon\mathbb{S}^2\to \mathcal Z_m$. Then there is a first-order twistor lift map $\mathcal T$, $V\mapsto l$ from the space of Jacobi fields along $\Phi$ to the space of twistor fields along $\Psi$ with a left inverse $\mathcal I$ given by $l\mapsto -2\Re(l^*(\Phi))$.
\end{proposition}

The goal of this section is to show that the first order twistor lift is a 1-to-1 correspondence. We will achieve it two steps: first we show that the left inverse $\mathcal I$ is in fact well-defined, i.e. for any twistor field $l$, $\mathcal I(l)$ is a Jacobi field. Second, we show that $\mathcal I$ is injective. We start with some general properties of twistor fields.

\begin{proposition}
\label{twistoradjoint:prop}
The field $l$ is twistor iff $l^*\in\hom(L_{\leqslant 0},L_{>0})$ satisfies the following properties,
\begin{itemize}
\item[1)] $l^*$ is holomorphic and
\item[2)] $(\nabla_z l^*)|_{L<0} = \partial_{z,0}\circ l^H$.
\end{itemize}
\end{proposition}
\begin{proof}
Let $v$ be a local section of $L_{<0}$ and $w$ be a local section of $L_{\leqslant 0}$. Then one has
\begin{equation*}
\begin{split}
((\nabla_{\bar z}l)(v),w)=(\nabla_{\bar z}(l(v)),w) - (l(\nabla_{\bar z}v),w) = \left (\frac{\partial}{\partial \bar z}(l(v)),w\right) - \left(l\left(\frac{\partial}{\partial \bar z}v\right),w\right),
\end{split}
\end{equation*}
since $(L_{-1},w)=0$ and $\frac{\partial}{\partial \bar z}v$ is a local section of $L_{<0}$. Continuing the computations, one has
\begin{equation*}
\begin{split}
((\nabla_{\bar z}l)(v),w) &=\frac{\partial}{\partial \bar z}(l(v),w) - \left(l(v), \frac{\partial}{\partial \bar z}w\right) - \left(\frac{\partial}{\partial \bar z}v,l^*(w)\right) \\ 
&= \frac{\partial}{\partial \bar z}(v,l^*(w))-\left(\frac{\partial}{\partial \bar z}v,l^*(w)\right)- \left(v, l^*(\frac{\partial}{\partial \bar z}w)\right) \\ 
&= \left(v, \frac{\partial}{\partial \bar z}(l^*(w))\right)- \left(v, l^*(\frac{\partial}{\partial \bar z}w)\right) = (v,(\nabla_{\bar z}l^*)(w)).
\end{split}
\end{equation*}
Thus, one concludes that $l$ is holomorphic iff $l^*$ is holomorphic.

Let $u,v$ be two local sections of $L_{<0}$. Using similar ideas as before, one has
\begin{equation*}
\begin{split}
&((\nabla_zl)(v),u) + ((l^*\circ\partial_{-1}(\partial_z)\circ\pi_{-1})v,u) = \left(\frac{\partial}{\partial z}(l(v)),u\right) - (l(\nabla^{<0}_z v),u) \\
&+ ((\partial_{z,-1}\circ\pi_{-1})v,l^H(u)) = \frac{\partial}{\partial z}(v,l^*(u)) - \left(l(v), \frac{\partial}{\partial z} u\right) - (\nabla^{<0}_z v,l^*(u)) \\
&+ \left(\frac{\partial}{\partial z} v, l^H(u)\right) = \frac{\partial}{\partial z}(v,l^*(u)) - (v,l^*(\nabla_z^{\leqslant 0} u)) - \left(\frac{\partial}{\partial z}v, l^*(u)\right) \\ 
&- \left(v, \frac{\partial}{\partial z} l^H(u)\right)=(v,(\nabla_z l^*)(u)) - (v,(\partial_{z,0}\circ l^H)(u))
\end{split}
\end{equation*}
\end{proof}

Let us clarify the geometric meaning of condition 2) in the definition of twistor fields. In~\cite{LW} the field $l$ along $\Psi$ is called {\em an infinitesimal horizontal deformation} if there is a family of maps $\Psi_t\to\mathcal Z_m$ such that $\Psi_0=\Psi$, $\frac{d}{dt}\big|_{t=0}\Psi_t =l$ and the vertical part of $d\Psi_t$ is $o(t)$ as $t\to 0$. We claim that this definition is equivalent to condition 2) of Proposition~\ref{twistoradjoint:prop}. Thus, twistor fields are precisely those fields along a horizontal holomorphic map $\Psi$ that preserve the properties of being horizontal and holomorphic up to the first order. In order to prove the claim let $L^t_{<0}$, $L_{>0}^t$ and $L^t_0$ be the subspaces $\Psi_t$, 
$\overline{\Psi}_t$ and $(\Psi_t\oplus\overline{\Psi}_t)^\perp$. The field $l = \frac{d}{dt}\big|_{t=0}\Psi_t$ is an infinitesimal horizontal deformation iff 
$$
\frac{\partial}{\partial z}\left( L^t_{<0} \right)\subset L^t_{<0}\oplus L^t_0 + o(t) = L^t_{\leqslant 0} + o(t).
$$
For $w\in L_{\leqslant 0}$ the value $l^*(w)$ is computed by taking a family $w_t\in L^t_{<0}$ and setting 
$$
l^*(w) = -\pi_{>0}\left(\frac{d}{dt}\Big|_{t=0} w_t\right).
$$ 
Let $v\in L_{<0}$ and let $v_t\in L^t_{<0}$ be a family. Then for the vector $\frac{\partial}{\partial z} v\in L_{\leqslant 0}$ there is a corresponding family $\frac{\partial}{\partial z} v_t$. Note that $\frac{\partial}{\partial z} v_t\in L^t_{\leqslant 0}$ for all $v$ iff $l$ is an infinitesimal horizontal deformation. Set $u_t = \pi^t_{>0}\frac{\partial}{\partial z} v_t$ be the projection onto $L^t_{>0}$, so that 
$$
l^*\left( \frac{\partial}{\partial z} v\right) = -\pi_{>0}\left(\frac{d}{dt}\Big|_{t=0}\left(\frac{\partial}{\partial z} v_t - u_t\right)\right).
$$
Finally we conclude,
\begin{equation*}
\begin{split}
&(\nabla_z l^*)(v) = -\frac{\partial}{\partial z}\pi_{>0}\left(\frac{d}{dt}\Big|_{t=0} v_t\right) +\pi_{>0}\left(\frac{d}{dt}\Big|_{t=0}\left(\frac{\partial}{\partial z} v_t - u_t\right)\right) \\
&= \partial_{z,0}\pi_0\frac{d}{dt}\Big|_{t=0} v_t - \pi_{>0}\left(\frac{d}{dt}\Big|_{t=0} u_t\right) = \partial_{z,0}l^H(v) - \pi_{>0}\left(\frac{d}{dt}\Big|_{t=0} u_t\right).
\end{split}
\end{equation*}
Rearranging the terms yields the claim.

%

\begin{theorem}
If $l$ is a twistor field, then 
 $$
 \pi_{>0}\left (\frac{\partial^2}{\partial z\partial {\bar z}} l^*(f_0)\right) = -\gamma_0 l^*(f_0).
 $$
 In particular, $\mathcal I(l)$ is a Jacobi field.
\end{theorem}

\begin{proof}
First, we claim that since $l^*(f_0)\in L_{>0}$ one has that  
$$
\pi_{>0}\left (\frac{\partial^2}{\partial z\partial {\bar z}}l^*(f_0) \right) = \nabla_z^{>0}\nabla_{\bar z}^{>0}l^*(f_0) - \partial_{z,0}\pi_0\frac{\partial}{\partial \bar z}l^*(f_0).
$$
Then, by properties $1)$ and $2)$ of the twistor field (since $\frac{\partial}{\partial\bar z}f_0\in L_{<0}$), one has
$$
\nabla_z^{>0}\nabla_{\bar z}^{>0}l^*(f_0) = \nabla_z^{>0} l^*\left(\frac{\partial}{\partial\bar z}f_0\right) = l^*\left(\frac{\partial^2}{\partial z\partial {\bar z}} f_0\right) + \partial_{z,0}l^H\left(\frac{\partial}{\partial\bar z}f_0\right).
$$

Finally, we claim that $\partial_{z,0}\pi_0\frac{\partial}{\partial \bar z}l^*(f_0) + \partial_{z,0}l^H(\frac{\partial}{\partial\bar z}f_0)=0$. Indeed, it is sufficient to show $\pi_0\frac{\partial}{\partial \bar z}l^*(f_0)+l^H(\frac{\partial}{\partial\bar z}f_0)=0$. Since both terms lie in $L_0$, it is sufficient to pair it with $f_0$.
\begin{equation*}
\begin{split}
 &\left(\pi_0\frac{\partial}{\partial \bar z}\rho^*(f_0), f_{0}\right) = \left(\frac{\partial}{\partial \bar z}\rho^*(f_0), f_0\right) = -\left(f_0, \rho^H\left(\frac{\partial}{\partial \bar z} f_0\right)\right).
\end{split}
\end{equation*}


Finally, the statement about $\mathcal I$ follows easily from Lemma~\ref{Jacobi:lemma}.

\end{proof}

\begin{proposition}
One has $\ker(\mathcal I)=0$.
\end{proposition}
\begin{proof}
Let $l$ be a twistor field such that $l^*(f_0) = 0$. Then $l^H = 0$ and the holomorphic nature of $l$ implies that for any local section $v$ of $L_{<0}$ one has $\frac{\partial}{\partial \bar z} l(v) = l\left(\frac{\partial}{\partial \bar z}\right)$.

We prove by induction on $p$ that for any local section $w$ on $L_{-p}$ one has $(w,l(v)) = 0$. The base $p=0$ is equivalent to $l^H=0$. To prove the step of the induction, we apply $\frac{\partial}{\partial \bar z}$ to $(w,l(v)) = 0$ to obtain
$$
0= \left(\frac{\partial}{\partial \bar z}w, l(v)\right) + \left(w, l\left(\frac{\partial}{\partial \bar z} v\right)\right) = (\bar\partial_{\bar z,p} w, l(v)),
$$
where we used the step of the induction in the last equality. Since $\bar\partial_{\bar z,p}$ is an isomorphism on $\Sigma'$, continuity implies that $(w',l(v))=0$ for any local section $w'\in \Gamma(L_{-p-1})$.
\end{proof}

As a result, we arrive at the following,
\begin{theorem}
\label{1stordertwistor:thm}
The linear map $\mathcal T$ is an isomorphism between Jacobi fields along a linearly full harmonic map $\Phi\colon \mathbb{S}^2\to\mathbb{S}^{2m}$ and twistor fields along its twistor lift $\Psi\colon \mathbb{S}^2\to \mathcal Z_m$.
\end{theorem}

\subsection{Action of $SO(2m+1,\mathbb{C})$} 
\label{soaction:sec}
The group $SO(2m+1,\mathbb{C})$ naturally acts on the twistor space $\mathcal Z_m$ by sending an isotropic subspace $V$ to an isotropic subspace $AV$. The action is holomorphic and preserves horizontal distribution. As a result, for any horizontal holomorphic map $\Psi\colon\mathbb{S}^2\to\mathcal Z_m$ and any $A\in SO(2m+1,\mathbb{C})$, the map $A\Psi$ is horizontal and holomorphic. The action on a linearly full harmonic map $\Phi\colon\mathbb{S}^2\to\mathbb{S}^{2m}$ is defined by the corresponding action on its twistor lift $\Psi\colon \mathbb{S}^2\to\mathcal Z_m$, i. e. $A\Phi = \pi_mA\Psi$. Note that $\Phi$ is linearly full iff $A\Phi$ is linearly full.

For $A\in SO(2m+1,\mathbb{C})$ the differential of the action defines the map $A_*\colon \Psi^*T^{(1,0)}\mathcal Z_m \to (A\Psi)^*T^{(1,0)}\mathcal Z_m$. In this section we prove that $A_*$ maps twistor fields to twistor fields and since $(A_*)^{-1}=(A^{-1})_*$, the application of Theorem~\ref{1stordertwistor:thm} implies Proposition~\ref{invariance:prop}.

Let $\{L_i\}$ be a harmonic sequence associated to linearly full map $\Phi$ and let $\{L^A_i\}$ be a harmonic sequence associated to $A\Phi$. We set $L^A_{<0} = \bigoplus_{i=-m}^{-1} L^A_i$ and, similarly use notations $L^A_{>0}$, $L^A_{\geqslant 0}$ and $L^A_{\leqslant 0}$ for objects constructed from the sequence $\{L^A_i\}$. By the exact form of the twistor correspondence, one has that $AL_{<0} = L^A_{<0}$. Since $\overline{L_{>0}}= L_{<0}$ and   
$\overline{L^A_{>0}}= L^A_{<0}$, one also has $\overline{A}L_{>0} = L^A_{>0}$. Furthermore, one has $AL_{\leqslant 0}\perp \overline{A}L_{>0}$ and, therefore, $AL_{\leqslant 0} = L^A_{\leqslant 0}$.

%
%

The map $A_*$ is easy to describe using the identification $\Psi^*T^{(1,0)}\mathcal Z_m$ with $\hom^s(L_{<0},L_{\geqslant 0})$. Since $AL_{<0} = L^A_{<0}$, under this identification $(A\Psi)^*T^{(1,0)}\mathcal Z_m$ becomes $\hom^s(AL_{<0}, (AL_{<0})^{\perp})$ and the map $A_*$ acts as
$$
(A_*l)(Av) = \pi^A_{\geqslant 0}(A(l(v))),
$$
where $\pi^A_{\geqslant 0}$ is the projection onto $L^A_{\geqslant 0}=(AL_{<0})^{\perp}$. The following proposition is a consequence of the fact that $SO(2m+1,\mathbb{C})$ acts holomorphically on $\mathcal Z_m$. We give a proof for completeness.
\begin{proposition}
The map $A_*$ preserves holomorphic fields.
\end{proposition}
\begin{proof}
Let $l$ be a holomorphic section of $\hom^s(L_{<0},L_{\geqslant 0})$ and $v$ be a local section of $L_{<0}$. Recall that $\pi^A_{\geqslant 0}\frac{\partial }{\partial \bar z}\pi^A_{\geqslant 0} = \pi^A_{\geqslant 0}\frac{\partial }{\partial \bar z}$. Then one has
\begin{equation*}
\begin{split}
&\nabla_{\bar z} ((A_*l)(Av)) = \pi^A_{\geqslant 0}\frac{\partial }{\partial \bar z}\pi^A_{\geqslant 0}(Al(v)) =  \pi^A_{\geqslant 0}\frac{\partial }{\partial \bar z}(Al(v)) = \pi^A_{\geqslant 0}A\frac{\partial }{\partial \bar z}l(v) \\
&= \pi^A_{\geqslant 0}Al\left(\frac{\partial }{\partial \bar z}v\right) = (A_*l)\left(A\frac{\partial }{\partial \bar z}v\right) = (A_*l)\left(\frac{\partial }{\partial \bar z}Av\right),
\end{split}
\end{equation*}
where in the third and the last equality we used that $A$ is a constant matrix, so it commutes with taking the derivatives.
\end{proof}

\begin{proposition}
The map $A_*$ preserves twistor fields.
\end{proposition}
\begin{proof}
Since $A_*$ preserves holomorphic fields, it is sufficient to check that the second condition in the definition of twistor fields is preserved. We split the proof into several steps.

First we establish the expression for $(A_*l)^*$. The claim is that for all local sections $w$ of $L_{<0}$ one has $(A^*l)^*(Aw) = \pi^A_{> 0}Al^*(w)$. Indeed, let $u$ be a local section of $L_{\leqslant 0}$. Then one has,
\begin{equation*}
\begin{split}
&((A^*l)^*(Aw), Au) = (Aw, (A_*l)(Au)) = (Aw,\pi^A_{>0}Al(u)) = (Aw, Al(u)) \\
&= (w,l(u)) = (l^*(w), u) = (Al^*(w), Au) = (\pi^A_{>0}Al^*(w), Au).
\end{split}
\end{equation*}

Let $u$ be a local section of $L_{<0}$. Then one has 
\begin{equation*}
\begin{split}
&\pi^A_{<0}\frac{\partial}{\partial z}Au - A\pi_{<0}\frac{\partial}{\partial z}u = \left(\frac{\partial}{\partial z}Au - \partial^A_{z,-1}\pi^A_{-1}Au\right) \\
&- \left( A\frac{\partial}{\partial z}u - A\partial_{z,-1}\pi_{-1}u\right) = A\partial_{z,-1}\pi_{-1}u - \partial^A_{z,-1}\pi^A_{-1}Au,
\end{split}
\end{equation*}
where we continue to use the superscript $A$ to denote objects associated to $A\Phi$.

Finally, we are in position to complete the proof. Suppose that $l$ satisfies the second condition in the definition in the twistor field and let $u$ be a local section of $L_{< 0}$. Then one has
\begin{equation*}
\begin{split}
&(\nabla^A_z(A_*l))^V(Au) = \pi^A_{>0}\frac{\partial}{\partial z}Al(u) - (A_*l)^V\left(\pi_{<0}\frac{\partial}{\partial z} Au\right) 
\\
&= \left(\pi^A_{>0} Al\left(\pi_{<0}\frac{\partial}{\partial z} u\right) - \pi^A_{>0} A l^*(\partial_{z,-1}\pi_{-1}u)\right) - \pi^A_{>0}Al\left(A^{-1}\pi_{<0}\frac{\partial}{\partial z} Au\right).
\end{split}
\end{equation*}
Furthermore, since $AL_{\leqslant 0} = L^A_{\leqslant 0}$, one has $\pi_{>0}Al^H=0$. Therefore, one can use Proposition~\ref{adjoint:prop}, item $1)$, to continue with the last term as follows
\begin{equation*}
\begin{split}
&\pi^A_{>0}Al\left(A^{-1}\pi^A_{<0}\frac{\partial}{\partial z} Au\right) = \pi^A_{>0}Al^V\left(A^{-1}\pi^A_{<0}\frac{\partial}{\partial z} Au\right)\\
& = -\pi^A_{>0}Al^*\left(A^{-1}\pi^A_{<0}\frac{\partial}{\partial z} Au\right)=-\pi^A_{>0}Al^*\left(\pi_{<0}\frac{\partial}{\partial z}u\right) - \pi^A_{>0}Al^*(\partial_{z,-1}\pi_{-1}u)\\
& + \pi^A_{>0}Al^*(A^{-1}\partial^A_{z,-1}\pi^A_{-1}Au)= \pi^A_{>0}Al\left(\pi_{<0}\frac{\partial}{\partial z}u\right) - \pi^A_{>0}Al^*(\partial_{z,-1}\pi_{-1}u)\\
& + (A_*l)^*(\partial^A_{z,-1}\pi^A_{-1}Au).
\end{split}
\end{equation*}
Putting it together with the previous equality yields,
$$
(\nabla^A_z(A_*l))^V(Au) = -(A_*l)^*(\partial^A_{-1}(\partial_z)\pi^A_{-1}Au),
$$
which is exactly the condition for $A_*l$.

\end{proof}

\end{document}